\documentclass{article}

\PassOptionsToPackage{numbers, compress}{natbib}

\usepackage[final]{neurips_2020}

\usepackage[utf8]{inputenc} 
\usepackage[T1]{fontenc}    
\usepackage[hidelinks]{hyperref}       
\usepackage{booktabs}       
\usepackage{amsfonts}       
\usepackage{nicefrac}       
\usepackage{microtype}      

\usepackage[capbesideposition=outside,capbesidesep=quad]{floatrow}
\usepackage{amssymb}
\usepackage{amsthm}
\usepackage{enumerate}

\newcommand{\R}{{\bf R}}

\newcommand{\argmin}{\mathop{\rm argmin}}

\newcommand{\argmax}{\mathop{\rm argmax}}

\newcommand{\ind}[1]{_{#1}}

\usepackage{physics}

\newtheorem{theorem}{Theorem}
\newtheorem*{theorem*}{Theorem}
\newtheorem{lemma}{Lemma}
\newtheorem{definition}{Definition}

\newtheorem{remark}{Remark}
\newtheorem{fact}{Fact}

\newtheorem{corollary}{Corollary}
\newtheorem{assumption}{Assumption}

\newcommand{\minimize}[2]{\underset{#1}{\text{minimize}} #2}
\newcommand{\maximize}[2]{\underset{#1}{\text{maximize}} #2}

\newcommand{\lag}{\mathcal{L}}
\newcommand{\sm}{L}

\newcommand{\maximum}[2]{\underset{#1}{\text{maximum}} #2}

\usepackage{algorithm}
\usepackage{algpseudocode}
\usepackage{graphicx}

\usepackage{color}
\usepackage{xcolor}

\newcommand{\stateSimplePsiMin}[0]{Simple-Psi-Minimization}
\newcommand{\callSimplePsiMin}[1]{\hyperref[alg:PGD]{\Call{\stateSimplePsiMin}{#1}}}

\newcommand{\stateEscapeExactLocalMin}[0]{Escape-Exact-Local-Min}
\newcommand{\callEscapeExactLocalMin}[1]{\hyperref[alg:escape-exact-local-min]{\Call{\stateEscapeExactLocalMin}{#1}}}

\newcommand{\stateSafePsiMin}[0]{Safe-Psi-Minimization}
\newcommand{\callSafePsiMin}[1]{\hyperref[alg:safe-PGD]{\Call{\stateSafePsiMin}{#1}}}

\newcommand{\stateFixDeg}[0]{Fix-Degeneracy}
\newcommand{\callFixDeg}[1]{\hyperref[alg:fix-deg]{\Call{\stateFixDeg}{#1}}}

\newcommand{\Forward}[1]{\Call{Forward}{#1}}

\newcommand{\Z}{Z}
\renewcommand{\S}{S}

\newcommand{\gain}[1]{t_{#1}}
\newcommand{\hatGain}[1]{\hat{t}_{#1}}
\newcommand{\totalGain}{\tau}

\def\fStar{f_{*}}

\usepackage{caption}
\usepackage{subcaption}

\usepackage{wrapfig}

\usepackage{todonotes}
\usepackage{comment}
\usepackage{multirow}

\usepackage[compact]{titlesec}
\titlespacing\section{2pt}{0pt plus 0pt minus 2pt}{0pt plus 1pt minus 1pt}
\titlespacing\subsection{2pt}{0pt plus 0pt minus 2pt}{1pt plus 0pt minus 2pt}
\titlespacing\subsubsection{2pt}{0pt plus 0pt minus 2pt}{1pt plus 0pt minus 2pt}
\titlespacing\paragraph{0pt}{0pt plus 0pt minus 2pt}{7pt plus 2pt minus 0pt}

\title{An efficient nonconvex reformulation of stagewise convex optimization problems}

\author{%
Rudy Bunel$^*$ \\
DeepMind \\
\texttt{rbunel@google.com} \\
\And
Oliver Hinder\thanks{Equal contribution} \\
Google Research, University of Pittsburgh \\
\texttt{ohinder@pitt.edu}
\AND
Srinadh Bhojanapalli \\
Google Research \\
\texttt{bsrinadh@google.com} \\
\And
Krishnamurthy (Dj) Dvijotham \\
DeepMind \\
\texttt{dvij@google.com}
}

\begin{document}

\maketitle
 
\begin{abstract}
    Convex optimization problems with staged structure appear in several contexts, including optimal control, verification of deep neural networks, and isotonic regression.
    Off-the-shelf solvers can solve these problems but may scale poorly.
    We develop a nonconvex reformulation designed to exploit this staged structure. Our reformulation has only simple bound constraints, enabling solution via projected gradient methods and their accelerated variants. 
    The method automatically generates a sequence of primal and dual feasible solutions to the original convex problem, making optimality certification easy.
    We establish theoretical properties of the nonconvex formulation, showing that it is (almost) free of spurious local minima and has the same global optimum as the convex problem.
    We modify PGD to avoid spurious local minimizers so it always converges to the global minimizer. For neural network verification, our approach obtains small duality gaps in only a few gradient steps. 
    Consequently, it can quickly solve large-scale verification problems  faster than both off-the-shelf and specialized solvers.
\end{abstract}

\section{Introduction}


This paper studies efficient algorithms for a particular class of stage-wise optimization problems:
\begin{subequations}\label{eq:convex-relaxation}
\begin{flalign}
& \minimize{(s, z) \in S \times \R^{n}}{ f(s, z)} \\
\text{s.t. }& \mu\ind{i}(s, z\ind{1:i-1}) \le z\ind{i} \le  \eta\ind{i}(s, z\ind{1:i-1})   & \forall i \in \{1, \dots, n \} \label{eq:convex-relaxation:constraints}
\end{flalign}
\end{subequations}
where $n$ and $m$ are positive integers, $S \subseteq \R^{m}$, the function $f$ has domain $S \times \R^{n}$ and range $\R$, the functions $\mu_i$ and $\eta_i$ have domain $S \times \R^{i-1}$ and range $\R$. Given a vector $z$, we use the notation $z_{1:i}$ to denote the vector $[z_1, \dots, z_i]$. We let $z_{1:0}$ be a vector of length zero. Throughout the paper we assume that $\eta\ind{1}, \dots, \eta_{n}$ are proper concave functions, $f$, $\mu\ind{1}, \dots, \mu\ind{n}$ are proper convex functions, and $S$ is a nonempty convex set.

Problems that fall into this problem class are ubiquitous. They appear in optimal control \cite{lewis2012optimal}, finite horizon Markov decision processes with cost function controlled by an adversary \cite{mcmahan2003planning}, generalized Isotonic regression \cite{gamarnik2019sparse,luss2012efficient},
and verification of neural networks \citep{bunel2018unified,Dvijotham2018, Wong2018provable}.
Details explaining how these problems can be written in the form of \eqref{eq:convex-relaxation} are given in Appendix~\ref{appendix:examples-with-this-structure}. Here we briefly outline how neural network verification falls into \eqref{eq:convex-relaxation:constraints}. Letting $s$ represent the input image and $z$ the activation values, neural networks verification can be written (unconventionally) as
$$
\minimize{(s, z) \in S \times \R^{n}}{ f(s, z)} \text{ s.t. } z_i = \sigma ( [s, z_{1:i-1}] \cdot w_i ),
$$
for (sparse) weight vectors $w_i$ and activation function $\sigma$.
A convex relaxation is created by picking functions satisfying $\mu_i(s, z_{1:i-1}) \le \sigma_i ( [s, z_{1:i-1}] \cdot w_i ) \le \eta_i(s, z_{1:i-1})$ for all $s$ and $z$ feasible to the original problem.
Solving these convex relaxations with traditional methods can be time consuming. 
For example, \citet{salman2019convex} reports spending 22 CPU years to solve problems of this type in order to evaluate the tightness of their proposed relaxation.
Consequently, methods for solving these relaxations faster are valuable. 
\subsection{Related work}

\subsubsection{Drawbacks of standard solvers for stagewise convex problems}\label{sec:drawbacks-of-existing-methods}
Standard techniques for solving \eqref{eq:convex-relaxation} can be split into two types: first-order methods and second-order methods. These techniques do not exploit this stage-wise structure, and so they face limitations.

\paragraph{First-order methods:} Methods such as mirror prox \cite{nemirovski2004prox}, primal-dual hybrid gradient (PDHG) \cite{chambolle2011first}, augmented lagrangian methods \cite{conn1996numerical}, and subgradient methods \cite{shor2012minimization} have cheap iterations (i.e., a matrix-vector multiply) but may require many iterations to converge. For example, 
\begin{flalign} \label{eq:hard-first-order-method-instance}
\minimize{x}{&-x_n} \quad \text{s.t.} \quad x_1 \in [0, 1], \quad -1 \le x_i \le x_{i-1} \quad \forall i \in \{1, \dots, n - 1\}
\end{flalign}
is an instance of \eqref{eq:convex-relaxation} with optimal solution at $x = \mathbf{1}$. However, this is the type of problem that exhibits the worst-case performance of a first-order method. In particular, one can show (see Appendix~\ref{appendix:lower-bound}) using the techniques of \citet[Section~2.1.2]{nesterov2013introductory} it will take at least $n-1$ iterations until methods such as PDHG or mirror-prox obtain an iterate with $x_1 > 0$ starting from $x = \mathbf{0}$. Furthermore, existing first-order methods are unable to generate a sequence of primal feasible solutions. This makes constructing duality gaps challenging. We could eliminate these constraints using a projection operator, but in general this will require calling a second-order method at each iteration, making iterations more expensive.

\paragraph{Second-order methods:} Methods such as interior point and simplex methods rely on factorizing a linear system, and can suffer from speed and memory problems on large-scale problems if the sparsity pattern is not amenable to factorization. This issue, for example, occurs in the verification of neural networks as dense layers force dense factorizations.

\subsubsection{Other nonconvex reformulations of convex problems}
Most research on nonconvex reformulations of convex problems is for semi-definite programs \cite{ boumal2016non, burer2003nonlinear, burer2005local}. In this work, the semi-definite variable is rewritten as the sum of low rank terms, forgoing convexity but avoiding storing the full semi-definite variable. 
Compared with this line of research our technique is unique for several reasons.
Firstly, our primary motivation is speed of convergence and obtaining certificates of optimality, rather than reducing memory or iteration cost. Secondly, the landscape of our nonconvex reformulation is different. For example, it contains spurious local minimizers (as opposed to saddle points) which we avoid via careful algorithm design.

\section{A nonconvex reformulation of stagewise convex problems}
We now present the main technical contribution of this paper, i.e., a nonconvex reformulation of the stagewise convex problems of the form \eqref{eq:convex-relaxation} and an analysis of efficient projected gradient algorithms applied to this formulation. 

\subsection{Assumptions}\label{sec:setup}
We begin by specifying assumptions we make on the objective and constraint functions in \eqref{eq:convex-relaxation}. Prior to doing so, it will be useful to introduce the notion of a smooth function:
\begin{definition}\label{def:smoothness}
A function $h : X \rightarrow \R$ is smooth if $\grad h(x)$ exists and is continuous for all $x \in X$; $h$ is $L$-smooth if $\| \grad h (x) - \grad h (x') \|_2 \le L \| x - x' \|_2, \forall x, x' \in X$.
\end{definition}

\begin{assumption}\label{assume:smooth}
Assume $f, \eta\ind{1}, \dots, \eta_{n}$, $\mu\ind{1}, \dots, \mu\ind{n}$ are smooth functions.
\end{assumption}

\begin{remark}\label{remark:smoothing-issue}
If Assumption~\ref{assume:smooth} fails to hold it is may be possible to approximate $f, \eta_i$ and $\mu_i$ by smooth functions \cite{nesterov2005smooth}. It is also possible one could use a nonsmooth optimization method \cite{davis2018stochastic}. However, we leave the study of these approaches to future work.
\end{remark}

Let $\Pi_{S}$ denote the projection operator   onto the set $S$. Ideally, the cost of this projection is cheap (e.g., $S$ is formed by simple bound constraints) as we will be running projected gradient descent (PGD) and therefore routinely using projections.

\begin{assumption}\label{assume:bounded-set}
Assume $S$ is a bounded set with diameter $D_s = \sup_{s,\hat{s} \in S} \| s - \hat{s} \|_2$. Further assume $Z$ is a bounded set such that for every feasible solution $(s,z)$ to \eqref{eq:convex-relaxation} we have $z \in Z$. Define $D_z = \sup_{z, \hat{z} \in Z} \| \hat{z} - z \|_2$.
\end{assumption}

 We remark that if $\eta$ and $\mu$ are smooth, and $S$ is bounded then there exists a set $Z$ satisfying Assumption~\ref{assume:bounded-set}. The primary reason for Assumption~\ref{assume:bounded-set} is it will allow us to form lower bounds on the optimal solution to \eqref{eq:convex-relaxation}. We will also find it useful to be able to readily construct upper bounds, i.e., feasible solutions to \eqref{eq:convex-relaxation}. This is captured by the following assumption.

\begin{assumption}\label{assume:inductive}
For all $i \in \{1, \dots, n \}$,
if $s \in S$ and $\mu_j(s, z_{1:j-1}) \le z_{j} \le \eta_j(s, z_{1:j-1})$ for $j \in \{1, \dots, i -1 \}$ then $\mu_i(s, z_{1:i-1}) \le \eta_i(s, z_{1:i-1})$.
\end{assumption}

 Assumption~\ref{assume:inductive} is equivalent to stating that feasible solutions to \eqref{eq:convex-relaxation} can be constructed inductively. In particular, given we have a feasible solution to the first $1, \dots, i - 1$ constraints we can find a feasible solution for the $i$th constraint by picking any $z_i \in [\mu_i(s, z_{1:i-1}), \eta_i(s, z_{1:i-1})]$ which must be a nonempty set by Assumption~\ref{assume:inductive}.
 All examples discussed in Appendix~\ref{appendix:examples-with-this-structure} satisfy Assumption~\ref{assume:inductive}.

\subsection{A nonconvex reformulation}

 Our idea is to apply PGD to the following nonconvex reformulation of \eqref{eq:convex-relaxation}, 
\begin{subequations}\label{eq:generic-nonconvex-reformulation}
\begin{flalign}
& \minimize{(s, z, \theta) \in S \times \R^{n} \times [0,1]^{n}}{f(s, z)} \\
\text{s.t. } & z\ind{i} = (1 - \theta_{i}) \mu\ind{i}(s, z\ind{1:i-1}) + \theta\ind{i} \eta\ind{i}(s, z\ind{1:i-1})  & \forall i \in \{1, \dots, n \}. \label{cons:nonconxex-xn-theta}
\end{flalign}
\end{subequations}
The basis of this reformulation is that if $\mu\ind{i}(s, z\ind{1:i-1}) \le z_i \le \eta\ind{i}(s, z\ind{1:i-1})$ then $z_i$ is a convex combination of $\mu\ind{i}(s, z\ind{1:i-1})$ and $\eta\ind{i}(s, z\ind{1:i-1})$.
This reformulation allows us to replace the $z$ variables with $\theta$ variables and replaces the constraints \eqref{eq:convex-relaxation:constraints} that are difficult to project onto with box constraints. 
For conciseness we denote \eqref{cons:nonconxex-xn-theta} by
$$
z \gets \Forward{s,\theta}.
$$
Let us consider an alternative interpretation of \eqref{eq:generic-nonconvex-reformulation} that explicitly replaces $z$ with $\theta$. Define $\psi\ind{n}(s, z) := f(s, z)$
and recursively define $\psi_i$ for all $i \in \{1, \dots, n \}$ by
\begin{flalign*}
\psi\ind{i-1}(s, z_{1:i-1}, \theta_{i:n}) &:= \psi\ind{i}(s, z_{1:i-1}, (1 - \theta_{i}) \mu_{i}(s, z_{1:i-1}) + \theta\ind{i} \eta_{i}(s, z_{1:i-1}), \theta_{i+1:n}).
\end{flalign*}
Note that $\psi_{i-1}$ eliminates the variable $z_i$ from $\psi_i$ by replacing it with $(1 - \theta_{i}) \mu_{i}(s, z_{1:i-1}) + \theta\ind{i} \eta_{i}(s, z_{1:i-1})$.
Using this notation, the reformulation \eqref{eq:generic-nonconvex-reformulation} is equivalent to:
\begin{flalign}
\minimize{(s, \theta) \in S \times [0, 1]^{n}}{\psi_0(s, \theta)}.
\end{flalign}
For intuition consider the following example
\begin{flalign}\label{example-values}
S := [-1,1], \quad f(s_1,z_1) := z_1, \quad \eta_1(s_1) := 1 - s_1^2, \quad \mu_1(s_1) := s_1^2 - 1.
\end{flalign}
In Figure~\ref{fig:introduce-reformulation} we plot this example. Consider an arbitrary feasible point, e.g., $z_1 = 0.0$, $s_1 = 0.5$ and note that point can be written as a convex combination of a point on $\eta$ and a point on $\mu$.
The nonconvex reformulation does this explicitly with box constraints replacing nonlinear constraints.

\begin{figure}[htbp]
    \centering
     \begin{tabular}{ c c }
    \includegraphics[height=110pt]{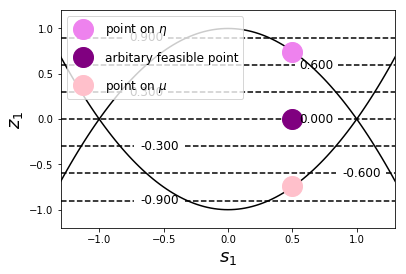} &
    \includegraphics[height=110pt]{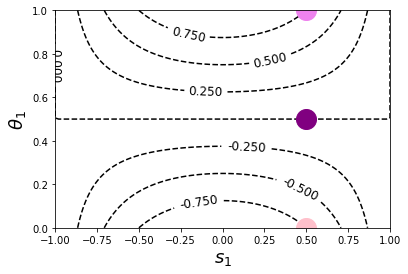} \\
    Plot of original convex problem & Plot of nonconvex reformulation $\psi_{0}(s_1, \theta_1)$
    \end{tabular}
    \caption{Comparison between original problem and reformulation.}  \label{fig:introduce-reformulation}
\end{figure}

The function $\psi_0$ is smooth (since it is the composition of smooth functions), and its gradient is computable by backpropagation, i.e., $\grad \psi_{n} = \grad f$ and for $i = n, \dots, 1$,
\begin{subequations}
\begin{flalign}
\grad_{s} \psi_{i-1} &= \grad_{s} \psi_{i} + \frac{\partial \psi_{i}}{\partial z_{i}} \left( \theta_{i} \grad_s \eta_{i} + ( 1 - \theta_{i}) \grad_s \mu_{i} \right) \label{eq:grad-psi-s} \\
\frac{\partial \psi_{i-1}}{\partial z_j} &=\frac{\partial \psi_{i}}{\partial z_j} + \frac{\partial \psi_{i}}{\partial z_{i}} \left( \theta_{i} \frac{\partial \eta_{i}}{\partial z_j} + ( 1 - \theta_{i}) \frac{\partial\mu_{i}}{\partial z_j} \right) \quad \forall j \in \{1, \dots, i - 1 \} \label{eq:grad-psi-z} \\
\frac{\partial \psi_0}{\partial \theta_i} &=
\frac{\partial \psi_i}{\partial \theta_i} = 
\frac{\partial \psi_{i}}{\partial z_{i}}\frac{\partial z_{i}}{\partial \theta_{i}} = 
\frac{\partial \psi_{i}}{\partial z_{i}}(\eta_i - \mu_i) 
\label{eq:grad-psi-theta}
\end{flalign}
\end{subequations}
where we denote $f = f(s, z)$,
$\psi_i = \psi_i(s, z_{1:i-1}, \theta_{i:n})$,
$\eta_i = \eta_i(s, z_{1:i-1})$, and
$\mu_i = \mu_i(s, z_{1:i-1})$;
this abuse of notation, where we assume these functions are evaluated at $(s, z, \theta)$ unless  specified otherwise, will continue throughout the paper for the purposes of brevity.
The subscript on $\grad$ specifies the arguments the derivative is with respect to, if it is left blank then we take the derivatives with respect to all arguments. Therefore, one can apply PGD, or other related descent algorithm to minimize $\psi_0$.
Moreover, the cost of computing the gradient via backpropagation is cheap (dominated by the cost of evaluating $\grad f$, $\grad \eta$, and $\grad \mu$). However, since $\psi_0$ is nonconvex, it is unclear whether a gradient based approach will find the global optimum. 

We show that this is indeed the case in the following subsections: In section \ref{sec:global}, we show that global minima are preserved under the nonconvex reformulation. In section \ref{sec:FO-global-in-nondegenerate-case}, show that \emph{nondegenerate} local optima are global optima and that projected gradient descent converges quickly to these. In section \ref{sec:fixing-the-degenerate-case}, we show how to modify projected gradient descent to avoid convergence to degenerate local optima and ensure convergence to a global optimum.

\subsection{Nonconvex reformulation is equivalent to original convex problem}\label{sec:global}
Before arguing that the local minimizers of \eqref{eq:generic-nonconvex-reformulation} are equal to the global minimizers of \eqref{eq:convex-relaxation}, it is important to confirm that the global minimizers are equivalent. Indeed, Theorem~\ref{lem:equivalence} confirms this.

\begin{theorem}\label{lem:equivalence}
Any feasible solution to \eqref{eq:convex-relaxation} corresponds to a feasible solution for \eqref{eq:generic-nonconvex-reformulation} with the same objective value.
Furthermore, if $\mu_i \le \eta_i$ for all $i \in \{1, \dots, n \}$ and $(s,z)$ feasible to \eqref{eq:generic-nonconvex-reformulation}, then any feasible solution to \eqref{eq:generic-nonconvex-reformulation} corresponds to a feasible solution for \eqref{eq:convex-relaxation} with the same objective value. In which case, the global optimum of \eqref{eq:generic-nonconvex-reformulation} is same as the global optimum of \eqref{eq:convex-relaxation}.
\end{theorem}
\begin{proof}
Consider any feasible solution $(s,z)$ to \eqref{eq:convex-relaxation}. By setting $\theta_i = \frac{z_i - \mu_i}{\eta_i - \mu_i}$ (any $\theta_i \in [0,1]$ suffices if $\mu_i = \eta_i$) we obtain a feasible solution to \eqref{eq:generic-nonconvex-reformulation}. On the other hand, if $\mu_i \le \eta_i$ then \eqref{cons:nonconxex-xn-theta} and $\theta_i \in [0,1]$ implies $\mu_i \le z_i \le \eta_i$.
\end{proof}
\begin{wrapfigure}[10]{r}{0.4\textwidth}
\vspace{-.9cm}
    \centering
    \includegraphics[width=\textwidth]{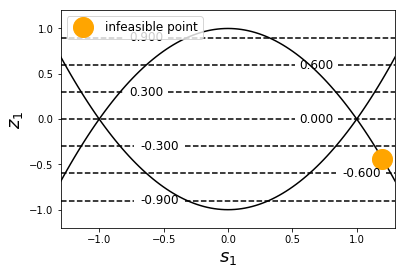}
    \caption{Infeasible: Assumption~\ref{assume:inductive} fails.}
        \label{fig:infeasibility-example}
\end{wrapfigure}

A sufficient condition for the premise of Theorem~\ref{lem:equivalence} to hold is Assumption~\ref{assume:inductive}. As Figure~\ref{fig:infeasibility-example} shows, if Assumption~\ref{assume:inductive} fails then the nonconvex reformulation can generate infeasible solutions to the original convex optimization problem \eqref{eq:convex-relaxation:constraints}. Consider the example given by \eqref{example-values} except with $S := [-1.5,1.5]$ instead of $S := [-1,1]$.
The set of feasible solutions to \eqref{eq:convex-relaxation} is enclosed by the two curves. At $s_1 = 1.2$ and $\theta = 1$, $\mu(s_1) > \eta(s_1)$, which is infeasible.

\subsection{Analysis of nondegenerate local optima}\label{sec:FO-global-in-nondegenerate-case}

This section is devoted to proving that under a nondegeneracy assumption, the first-order stationary points of \eqref{eq:generic-nonconvex-reformulation} are global minimizers.
Degeneracy issues arise when $\eta_i = \mu_i$.
In this situation, if $\theta_i$ changes, then $z$ will remain the same, and therefore from the perspective of the convex formulation, the solution is the same.
However, from the perspective of the function $\psi_0$ there is an important difference.
In particular, as $\theta_i$ changes the gradient of $\psi_0$ changes.
Consequently, certain values of $\theta_i$ may generate spurious local minimizers.
Recall example \eqref{example-values}, i.e., $S := [-1,1]$, $f(s_1,z_1) := z_1$, $\eta_1(s_1) := 1 - s_1^2$ and $\mu_1(s_1) := s_1^2 - 1$. In this instance,
\begin{flalign*}
\psi_0 = \theta_1 (1 - s_1^2) + (1 - \theta_1) (s_1^2 - 1) = (1 - 2 \theta_1) (s_1^2 - 1), \quad \frac{\partial \psi_0}{\partial s_1} = (1 - 2 \theta_1) (2 s_1 - 1).
\end{flalign*}
As illustrated in Figure~\ref{fig:degenerancy-psi-example}, the global minimizer is $s_1 = 0, \theta = 0 \Rightarrow z_1 = -1$. If $s_1 \pm 1$ then for all $\theta_1 \in [0,1]$ we have $z_1 = 0$. Moreover, the points $s_1 \pm 1$, $\theta_1 \in (0.5,1]$ are spurious local minimizers. To see this, note for all $\theta_1 \in [0.5,1]$, and $s_1 \in S$ that $\psi_0(s_1, \theta_1) \ge 0 = \psi_0(1, \theta_1)$. In contrast, the points $s_1 \pm 1$, $\theta \in [0, 0.5)$ are \emph{not} local minimizers, because for $s_1 \pm 1$ and $\theta_1 \in [0,0.5)$ we have $\frac{\partial \psi_0}{\partial s_1} > 0$ implying that gradient descent steps move away from the boundary. We conclude that if $\mu_i = \eta_i$ \emph{certain} values of $\theta_i$ could be spurious local minimizers. We emphasize the word \emph{certain} because, as  Section~\ref{sec:fixing-the-degenerate-case} details, there is always a value of $\theta_i$ that enables escape.

The nondegeneracy assumption we make is that for some $\gamma \ge 0$ the set
\begin{flalign*}
\mathcal{K}_{\gamma}(s, \theta) := \bigg\{ & i \in \{1, \dots, n \} : \quad z = \text{Forward}(s, \theta), \\
& \quad \eta_i - \mu_i \le \gamma, \quad \theta_i  \left( \frac{\partial \psi_i}{\partial z_i} \right)^{+} + (1 - \theta_i) \left( \frac{\partial \psi_i}{\partial z_i}\right)^{-} > 0 \bigg\}
\end{flalign*}
is empty, where $(\cdot)^{+} := \max\{ \cdot, 0 \}$ and $(\cdot)^{-} := \min\{ \cdot, 0 \}$. 
If the set $\mathcal{K}_{0}(s, \theta)$ is non-empty then any coordinate $i \in \mathcal{K}_{0}(s, \theta)$ could be causing a spurious local minimizer.
Values of $\gamma$ strictly greater than zero ensures that we do not get arbitrarily close to a degenerate point.
We will show this nondegeneracy assumption guarantees that first-order stationary points are global minimizers.

\begin{figure}[htbp]
    \centering
     \begin{tabular}{ c c }
    \includegraphics[height=110pt]{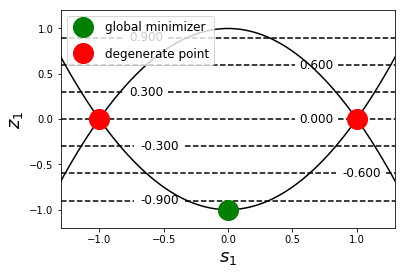} &
    \includegraphics[height=110pt]{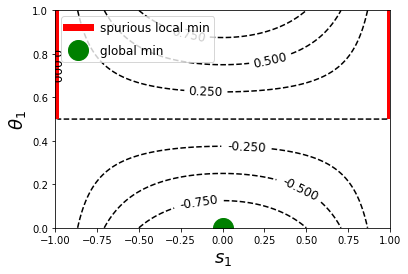} \\
    Plot of original convex problem & Plot of nonconvex reformulation $\psi_{0}(s_1, \theta_1)$
    \end{tabular}
    \caption{Example of degeneracy causing spurious local minimizers  when $s_1 \pm 1$.}  \label{fig:degenerancy-psi-example}
\end{figure}

While our nondegeneracy assumption holds it will suffice to run PGD which is defined as
$$
(s^{+},\theta^{+}) \gets (s,\theta) + \argmin_{d \in \mathcal{D}(s,\theta)}{\grad \psi_0 \cdot d + \frac{L}{2} \| d \|_2^2},
$$
where $\mathcal{D}(s,\theta) := \{ d : (s, \theta) + d \in S \times [0,1]^{n} \}$ is the set of feasible search directions and $L$ is the smoothness of $\psi_0$ (see Definition~\ref{def:smoothness}). A useful fact is that PGD satisfies
$\psi_0(s^{+},\theta^{+}) \le \psi_0(s, \theta) - \delta_{L}(s, \theta)$ for
\begin{flalign*}
\delta_{\sm}(s,\theta) := -\minimize{d \in \mathcal{D}(s,\theta)}{~\grad \psi_0 \cdot d + \frac{\sm}{2} \| d \|_2^2}.
\end{flalign*}
See \cite[Lemma 2.3.]{beck2009fast} for a proof. In other words, $\delta_L(s, \theta)$ represents the minimum progress of PGD. Once again for brevity we will denote $\delta_{\sm}(s,\theta)$ by $\delta_L$.
Note that if $\delta_L$ is zero then we are at a first-order stationary point of $\psi_0$. The remainder of this section focuses on proving that $\delta_L$ provides an upper bound on the optimality gap.
To form this bound we use Lagrangian duality. In particular, the Lagrangian of \eqref{eq:convex-relaxation} is:
\begin{flalign*}
\mathcal{L}(s, z, y) := f + \sum_{i=1}^{n} ( y\ind{i}^{+} \mu\ind{i} - y\ind{i}^{-} \eta\ind{i} - y\ind{i} z_{i})
\end{flalign*}
where $y^{+}_i = \max\{y_i, 0\}$, and $y^{-}_i = \max\{-y_i, 0 \}$. 
 We will denote $\lag(s, z, y)$ by $\lag$. Define,
\begin{flalign}\label{eq:form-dual-bound}
\Delta(s,\theta) := \sum_{i=1}^{n} ( y\ind{i} z_{i} - y\ind{i}^{+} \mu\ind{i} + y\ind{i}^{-} \eta\ind{i}) + \sup_{(\hat{s},\hat{z}) \in \S \times \Z} \grad_{s,z} \lag \cdot (\hat{s} - s, \hat{z} - z)
\end{flalign}
with $z = \Forward{s,\theta}$ and  $y_i = \frac{\partial \psi_i}{\partial z_i}$. If $(s,z)$ is feasible to \eqref{eq:convex-relaxation} we conclude $\Delta(s,\theta)$ is a valid duality gap, i.e., provides global guarantees, because by duality, convexity and \eqref{eq:form-dual-bound},
\begin{flalign}
\label{eq:valid-duality-gap}
\fStar \ge \inf_{(\hat{s}, \hat{z}) \in S \times Z}\lag(\hat{s},\hat{z},y) \ge \lag + \inf_{(\hat{s},\hat{z}) \in \S \times \Z} \grad_{s,z} \lag \cdot (\hat{s} - s, \hat{z} - z) = f - \Delta(s,\theta).
\end{flalign}
To compute  $\Delta(s,\theta)$, one needs to be able to efficiently minimize a linear function over the set $Z$. For this reason, one should choose $Z$ to have a simple form (i.e., bound constraints).

\begin{assumption}
\label{assume:V-W}
There exists a constant $c > 0$ such that 
$\| \eta - \mu \|_2 + D_s \| \grad_s \eta - \grad_s \mu \|_2 + D_z \| \grad_z \eta - \grad_z \mu \|_2 \le c$ for all $(s,z)$ that are feasible to \eqref{eq:convex-relaxation:constraints}.
\end{assumption}

In Assumption~\ref{assume:V-W}, observe that $\grad_s \eta - \grad_s \mu$ and $\grad_z \eta - \grad_z \mu$ are matrices so  $\| \cdot \|_2$ is the spectral norm.
Also, note that Assumption~\ref{assume:smooth} and~\ref{assume:bounded-set} imply that Assumption~\ref{assume:V-W} must hold. However, we add Assumption~\ref{assume:V-W} because it makes the constant $c$ explicit.


\begin{lemma}[Nondegenerate first-order stationary points are optimal]
\label{lem:progress-relationship-to-gap}
Suppose Assumption~\ref{assume:bounded-set} and \ref{assume:V-W} hold.
Suppose also that $\mathcal{K}_{\gamma}(s,\theta) = \emptyset$ with $\gamma \in (0, \infty)$, and that $\delta_{\sm} \le L / 2$. 
Then $\Delta(s,\theta)^2 \le L \left( D_s \sqrt{2} + 2 \gamma^{-1} c \right)^2  \delta_L$.
\end{lemma}

In the nondegenerate case (i.e., $\mathcal{K}_{\gamma}(s, \theta) = \emptyset$), $\delta_L$ upper bounds $\Delta(s,\theta)$. In particular, as Lemma~\ref{lem:progress-relationship-to-gap} demonstrates small progress by gradient steps implies small duality gaps. The proof of Lemma~\ref{lem:progress-relationship-to-gap} appears in Section~\ref{sec:proof-of:lem:progress-relationship-to-gap} and is technical. The core part of the proof of Lemma~\ref{lem:progress-relationship-to-gap} is bounding $\theta_i y^{+}_i +  (1 - \theta_i) y^{-}_i$ for $y_i = \frac{\partial \psi_i}{\partial z_i}$ in terms of $\gamma^{-1}$ and $\delta_L$. When $\theta_i y^{+}_i +  (1 - \theta_i) y^{-}_i \approx 0$ one can show that $\Delta(s,\theta) \approx  \sup_{\hat{s} \in S} \grad_s \lag \cdot (\hat{s} - s) \approx \sup_{\hat{s} \in S} \grad_s \psi \cdot (\hat{s} - s) \le D_s \sqrt{2 L \delta_L}$.

\subsubsection{Analysis of projected gradient descent}

Lemma~\ref{lem:progress-relationship-to-gap} provides the tool we need to prove the convergence of PGD in the nondegenerate case. The algorithm we analyze (Algorithm~\ref{alg:PGD}) includes termination checks for optimality. Furthermore, the PGD steps can be replaced by any algorithm that makes at least as much function value reduction as PGD would make in the worst-case.
For example, gradient descent with a backtracking line search and an Armijo rule \cite[Chapter~3]{nocedal2006numerical}, or 
a safeguarded accelerated scheme \cite{li2015accelerated} would suffice.

\begin{algorithm}[H]
\caption{Local search algorithm for minimizing $\psi_0$ in the nondegenerate case.}\label{alg:PGD}
\begin{algorithmic}[1]
\Function{\stateSimplePsiMin}{$s^{1}, \theta^{1},  \epsilon$}
\State Suppose $\psi_0$ is $L$-smooth. Note $L \in (0,\infty)$ need not be known.
\For{$k = 1, \dots, \infty$}
\State \textit{Termination checks:}
\If{$\Delta(s^{k}, \theta^{k}) \le \epsilon$}
\State \textit{Found an $\epsilon$-optimal solution:}
\State \Return $(s^{k}, \theta^{k})$
\EndIf
\State \textit{Reduce the function at least as much as PGD would:}
\State $(s^{k+1}, \theta^{k+1}) \in \{(s,\theta) : \psi_0(s, \theta) \le \psi_0(s^{k}, \theta^{k}) - \delta_{\sm}(s^{k}, \theta^k) \}$ \label{algline:optstep}
\EndFor
\EndFunction
\end{algorithmic}
\end{algorithm}

\begin{theorem}[PGD converges to global minimizer under nondegeneracy assumption]
\label{thm:simple-psi-minimization}
Suppose Assumption~\ref{assume:bounded-set}, \ref{assume:inductive} and \ref{assume:V-W} hold. Suppose $\psi_0$ is $L$-smooth, $\epsilon, \gamma \in (0,\infty)$, $(s^{1}, \theta^{1}) \in S \times [0,1]^{n}$, and $\mathcal{K}_{\gamma}(s^k,\theta^k) = \emptyset$ for all iterates of the algorithm \callSimplePsiMin{$s^{1}, \theta^{1}, \epsilon$}. Then, the algorithm terminates after at most
$$
1 + \frac{2 \Delta(s^{1}, \theta^{1})}{L} + 
\frac{\sm \left( D_s \sqrt{2} + 2 c \gamma^{-1} \right)^2}{\epsilon} \quad 
\text{iterations.}
$$
\end{theorem}

See Section~\ref{sec:simple-psi-minimization-proof} for a proof of Theorem~\ref{thm:simple-psi-minimization}. The proof of Theorem~\ref{thm:simple-psi-minimization} directly utilizes Lemma~\ref{lem:progress-relationship-to-gap} using standard techniques, almost identical to the proof of convergence for gradient descent in the convex setting \cite[Theorem~2.1.13]{nesterov2013introductory}.

\begin{remark}\label{remark:smoothness}
It is worth discussing the premise in Theorem~\ref{thm:simple-psi-minimization} that $\psi_0$ is $L$-smooth.
The composition of smooth functions is smooth, implying $\psi_0$ is smooth.   Moreover, since $S \times [0,1]^{n}$ is a bounded set we deduce that $\psi_0$ is $L$-smooth for some $L > 0$. Therefore the premise that $\psi_0$ is $L$-smooth is valid.
However, the value of $L$ could be extremely large, for example, if $\eta_i(s, z_{1:i-1}) = \mu_i(s,z_{1:i-1}) = 2 z_{i-1}$ for $i > 1$, $\eta_1(s) = \mu_1(s) = s_1$, and $f(s,z) = \frac{1}{2} z_n^2$ then
$\psi_0(s,\theta) = \frac{1}{2} (2^n s)^2$ and $L = 4^{n}$. Note this occurs despite the fact that each component function is well-behaved (i.e., $\eta_i, \mu_i$, $f$ are $1$-smooth and $2$-Lipschitz with respect to the Euclidean norm).
\end{remark}

\begin{remark}
Consider \eqref{eq:hard-first-order-method-instance}, the hard example for standard first-order methods. Note that starting from the origin (i.e., $x_1 = 0$, $\theta = \mathbf{0}$), then for sufficiently large step size PGD on $\psi_0$ will take exactly one iteration to find the optimal solution ($x_1 = 1, \theta = \mathbf{1}$).
\end{remark}

\begin{remark}
Suppose that we are solving a neural network verification problem (Section~\ref{sec:experiements} and ~\ref{sec:certification-of-neural-networks}). 
Then this approach is strongly related to adversarial attack heuristics.
In particular, freezing $\theta = \mathbf{0}$ in \callSimplePsiMin{} yields a typical gradient based attack on the network \cite{szegedy2013intriguing}.
\end{remark}

\subsection{Analysis of degenerate local optima}\label{sec:fixing-the-degenerate-case}

 Section~\ref{sec:FO-global-in-nondegenerate-case} proved convergence of PGD to the global minimizer under a nondegeneracy assumption (i.e., $\mathcal{K}_{\gamma}(s^k, \theta^k) = \emptyset$). 
This section develops a variant of PGD that requires no degeneracy assumptions but still converges to the global minimizer.

\subsubsection{Escaping exact local minimizers}

Our main result, presented in Section~\ref{subset:escaping-basins-of-local-minimizers}, proves convergence under minimal assumptions. The key to the result is developing an algorithm for escaping basins of local minimizers. However, the algorithm and analysis is very technical. To give intuition for it this section considers the easier case of escaping \emph{exact} local minimizers (Lemma~\ref{lem:Escaping-Exact-Local-Minimizer}). 

The high level idea is illustrated in Figure~\ref{fig:escaping-exact-local-minimizers}. Recall from Figure~\ref{fig:degenerancy-psi-example} that if we are at a spurious local minimizer then the set $\mathcal{K}_{\gamma}(s,\theta)$ must be nonempty. In particular, in this instance the set $\mathcal{K}_0(s,\theta) = \{1\}$ is nonempty. In this setting, $\theta_1$ corresponds to an edge that we can move along where $\psi_0(s,\theta)$ is constant. \callEscapeExactLocalMin{$s,\theta
$} moves us along this edge from $(s,\theta)$ to $(s,\hat{\theta})$ at which $\mathcal{K}_0(s,\hat{\theta})$ is empty and therefore we have escaped the local minimizer.

\begin{figure}[htbp]
\centering
\begin{subfigure}[b]{.5\textwidth}
  \centering
  \begin{algorithmic}[1]
\Function{\stateEscapeExactLocalMin}{$s, \theta$}
\State $z = \Forward{s, \theta}$, $\hat{\theta} \gets \text{copy}(\theta)$
\For{$i = n, \dots, 1$}
\If{$i \in \mathcal{K}_0(s,\theta_{1:i}, \hat{\theta}_{i+1:n})$}
\State $
\hat{\theta}_i = \begin{cases}
0 & \frac{\partial \psi_i(s, z_{1:i}, \hat{\theta}_{i+1:n})}{\partial z_i} > 0 \\
1 & \frac{\partial \psi_i(s, z_{1:i}, \hat{\theta}_{i+1:n})}{\partial z_i} < 0
\end{cases}
$
\EndIf
\EndFor
\State \Return $(s,\hat{\theta})$
\EndFunction
\end{algorithmic}
  \caption{Algorithm}
  \label{alg:escape-exact-local-min}
\end{subfigure}%
\begin{subfigure}[b]{.5\textwidth}
  \centering
  \includegraphics[width=0.9\linewidth]{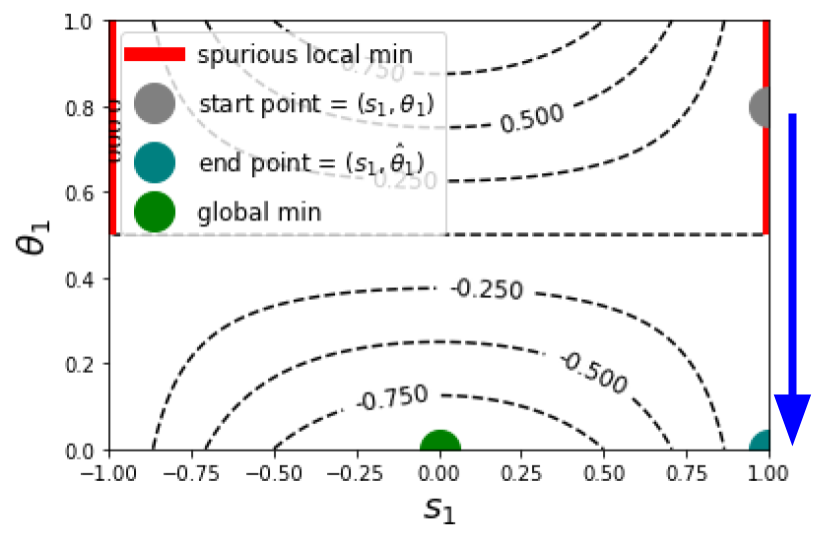}
  \caption{The high level idea of the algorithm is shown by the blue arrow.}
  \label{fig:escaping-exact-local-minimizers}
\end{subfigure}
\caption{An algorithm for escaping exact local minimizers}
\end{figure}

\begin{lemma}[Escaping exact local minimizers]
\label{lem:Escaping-Exact-Local-Minimizer}
Suppose that Assumption~\ref{assume:smooth} holds and let $(s,\hat{\theta}) = \callEscapeExactLocalMin{s, \theta}$.
Then
$\Forward{s,\theta} = \Forward{s, \hat{\theta}}$, and $\mathcal{K}_0(s,\hat{\theta}) = \emptyset$.
\end{lemma}
\begin{proof}
By the definition of $\mathcal{K}_0$, if $i \in \mathcal{K}_0(s,\theta)$ then $\eta_i = \mu_i$. Therefore
$\Forward{s, \theta_{1:i-1}, \hat{\theta}_{i:n}} = \Forward{s, \theta_{1:i}, \hat{\theta}_{i+1:n}}$, and by induction $ \Forward{s,\theta} = \Forward{s, \hat{\theta}}$.

Next, we show that $i \not\in \mathcal{K}_0(s, \theta_{1:i-1},\hat{\theta}_{i:n})$. If $i \not\in \mathcal{K}_0(s, \theta_{1:i},\hat{\theta}_{i+1:n})$ then $\theta_i = \hat{\theta}_i$ so the result trivially holds. On the other hand, if $i \in \mathcal{K}_0(s, \theta_{1:i},\hat{\theta}_{i+1:n})$ then by definition of $\hat{\theta}_i$,
\begin{flalign}\label{eq:theta-hat-psi-is-zero}
\hat{\theta}_i  \left( \frac{\partial \psi_i(s, z_{1:i}, \hat{\theta}_{i+1:n})}{\partial z_i} \right)^{+} + (1 - \hat{\theta}_i) \left( \frac{\partial \psi_i(s, z_{1:i}, \hat{\theta}_{i+1:n})}{\partial z_i}\right)^{-} = 0
\end{flalign}
which implies $i \not\in \mathcal{K}_0(s, \theta_{1:i-1},\hat{\theta}_{i:n})$. Further note that $\Forward{s, \theta_{1:i-1}, \hat{\theta}_{i:n}} = \Forward{s, \theta_{1:i}, \hat{\theta}_{i+1:n}}$
 implies if $j \le i$ and $j \not\in \mathcal{K}_0(s, \theta_{1:j-1},\hat{\theta}_{j:n})$ then $j \not\in \mathcal{K}_0(s, \theta_{1:i-1},\hat{\theta}_{i:n})$. By induction we deduce $\mathcal{K}_0(s, \theta_{1:i-1},\hat{\theta}_{i:n}) \subseteq \{ 1, \dots, i - 1 \}$ and hence $\mathcal{K}_0(s, \hat{\theta})$ is empty.
\end{proof}

A critical feature of \callEscapeExactLocalMin{$s,\theta
$} is that we work backwards (i.e., $i = n, \dots, 1$ rather than $i = 1, \dots, n$). This is critical because if we work forwards instead of backwards then \eqref{eq:theta-hat-psi-is-zero} would become
$$\hat{\theta}_i  \left( \frac{\partial \psi_i(s, z_{1:i}, \theta_{i+1:n})}{\partial z_i} \right)^{+} + (1 - \hat{\theta}_i) \left( \frac{\partial \psi_i(s, z_{1:i}, \theta_{i+1:n})}{\partial z_i}\right)^{-} = 0$$
which, due to the replacement of $\theta$ with $\hat{\theta}$ inside $\psi_i$, is insufficient to establish $\mathcal{K}_0(s,\hat{\theta})$ is empty.

Finally, we remark that $g_i := \frac{\partial \psi_i(s, z_{1:i}, \hat{\theta}_{i+1:n})}{\partial z_i}$ can be computed via the recursion
$$
g_i \gets \frac{\partial f}{\partial z_{i}} + \sum_{j = i + 1}^{n} g_{j} \left( \hat{\theta}_{j} \frac{\partial \eta_{j}}{\partial z_i} + ( 1 - \hat{\theta}_{j}) \frac{\partial \mu_{j}}{\partial z_i} \right),
$$
and therefore calling \callEscapeExactLocalMin{} takes the same time as computing $\grad_{\theta} \psi_0$.

\subsubsection{Escaping the basin of a local minimizer}\label{subset:escaping-basins-of-local-minimizers}

If we modify \callSimplePsiMin{} to run \callEscapeExactLocalMin{$s,\theta
$} whenever the set $\mathcal{K}_{0}(s^k,\theta^k)$ is nonempty then we would escape exact local minimizers. However, that does not exclude the possibility of asymptotically converging to a local minimizer. Therefore we need a method that will escape the basin of a local minimizer. In particular, we must be able to change the value of the $\theta_i$ variables with $i \in \mathcal{K}_{\gamma}(s,\theta)$ for $\gamma > 0$. This, however, introduces technical complications because if $\eta_i > \mu_i$ then as we change $\theta_i$ the value of $z_{i:n}$ could change.

Due to these technical complications we defer the algorithm and analysis to Appendix~\ref{app:escaping-the-basin-of-a-local-minimizer}, and informally state the main result here. The proof of Theorem~\ref{thm:general-convergence-result} appears in Appendix~\ref{app:proof-of-general-convergence-result}. The discussion given in Remark~\ref{remark:smoothness} also applies to Theorem~\ref{thm:general-convergence-result} and means that the constant $C$ could be large.

\begin{theorem}
\label{thm:general-convergence-result}
Suppose that Assumptions~\ref{assume:smooth}, \ref{assume:bounded-set}, and \ref{assume:inductive} hold. Then there exists an algorithm obtaining an $\epsilon$-duality gap after $C \epsilon^{-3} + 1$ computations of $\grad \psi_0$ where $C$ is a problem dependent constant.
\end{theorem}

\section{Experiments}\label{sec:experiements}

We evaluate our method on robustness verification of models trained on CIFAR10 \cite{krizhevsky2009learning}. We benchmark on three sizes of networks trained with adversarial training~\citep{madry2018towards}. The tiny network has two fully connected layers, with 100 units in the hidden layer. The small network has two convolutional layers, and two fully connected layers, with a total of 8308 hidden units. The medium network has four convolutional layers followed by three fully connected layers, with a total of 46912 hidden units. 
\begin{table}[b]
\vspace{3pt}
\floatbox[{\capbeside\thisfloatsetup{floatwidth=sidefil,capbesideposition={right,top},capbesidewidth=.31\linewidth}}]{table}
{\small \caption{\textbf{Benchmark}\\ For each model, we report the average bound achieved on the adversarial objective and the average runtime in milliseconds to obtain it, over the CIFAR-10 test set. IBP \citep{Gowal2019} does not perform any optimization so it has an extremely small runtime but the bounds it generates are much weaker.
The off-the-shelf solvers are significantly slower than the first-order methods DeepVerify and NonConvex and were not feasible to run beyond the tiny network.}\label{tab:benchmark}}%
{\small\addtolength{\tabcolsep}{-2.2pt}\raggedright
\begin{tabular}{l@{\hskip 0in}ccc@{\hskip 0in}ccc}
\toprule
  \textbf{ReLU Activation} & \multicolumn{3}{c}{\textbf{Average Bound}} &\multicolumn{3}{c}{\textbf{Runtime (ms)}}\\
  \cmidrule(r{1em}){2-4} \cmidrule{5-7}
   & Tiny & Small & Medium & Tiny & Small & Medium\\
   \midrule
   IBP~\citep{Gowal2019} & 17.0 & 743 & 2.4e+6 & 5.5 & 3.1 & 3.3 \\
   \midrule
   DeepVerify~\citep{Dvijotham2018} & 13.7 & 544 & 1.6e+6 & 349 & 711 & 1.1e+3 \\
   NonConvex (Ours) & 5.68 & 434.9 & 1.5e+6 & 91.2 & 177 & 175 \\
   \midrule
   CVXPY (SCS)~\citep{Odonoghue2016} & 5.64 & - & - & 1.7e+5 & - & -\\
   CVXPY (ECOS)~\citep{domahidi2013ecos} & 5.64 & - & - & 4.3e+4 & - & -\\
  \bottomrule
\toprule
  \textbf{SoftPlus Activation} & \multicolumn{3}{c}{\textbf{Average Bound}} &\multicolumn{3}{c}{\textbf{Runtime (ms)}}\\
  \cmidrule(r{1em}){2-4} \cmidrule{5-7}
   & Tiny & Small & Medium & Tiny & Small & Medium\\
   \midrule
   IBP~\citep{Gowal2019} & 18.3 & 6.5e+3 & 2.0e+9 & 4 & 2.5 & 3.3 \\
   \midrule
   DeepVerify~\citep{Dvijotham2018} & 13.7 & 5.1e+3 & 1.5e+9 & 414 & 855 & 1.7e+3\\
   NonConvex (Ours) & 5.97 & 3.93e+3 & 1.3e+9 & 7.8 & 65 & 214\\
   \midrule
   CVXPY (SCS)~\citep{Odonoghue2016} & 5.97 & - & - & 2.9e+5 & - & -\\
  \bottomrule
\end{tabular}}
\end{table}

\begin{figure}
    \begin{subfigure}{.62\textwidth}
        \includegraphics[width=\linewidth]{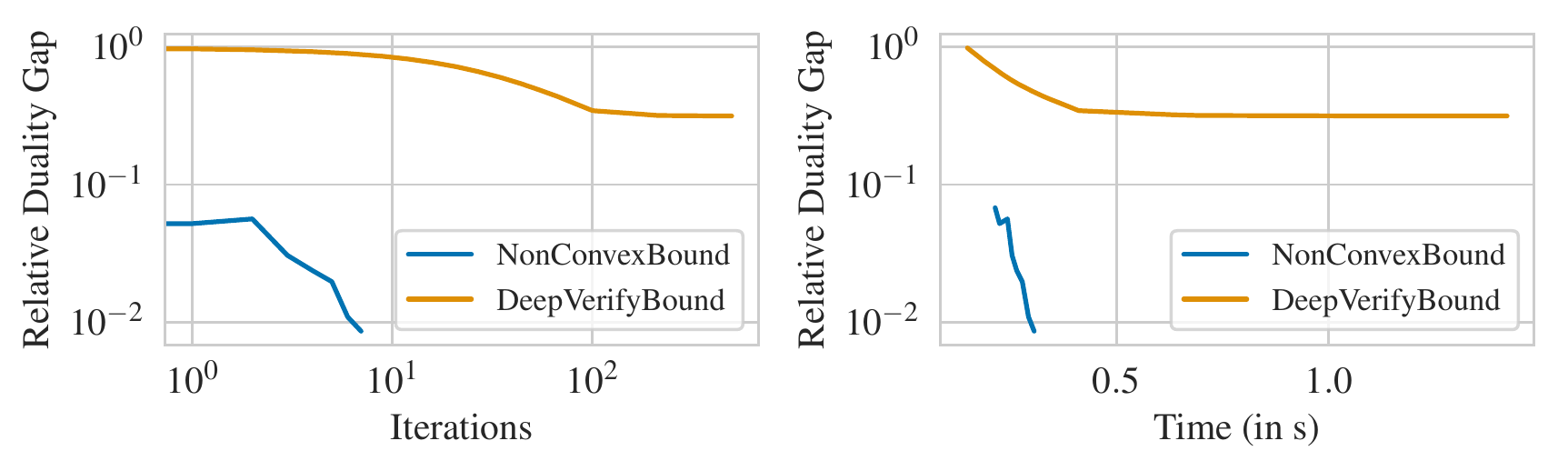}%
        \caption{Evolution of the relative duality gap as a function of time or number of iteration, for the NonConvex and DeepVerify Solver.}
        \label{fig:opt_trace}
    \end{subfigure}%
    \hfill%
    \begin{subfigure}{.31\textwidth}
        \centering
        \includegraphics[width=\linewidth]{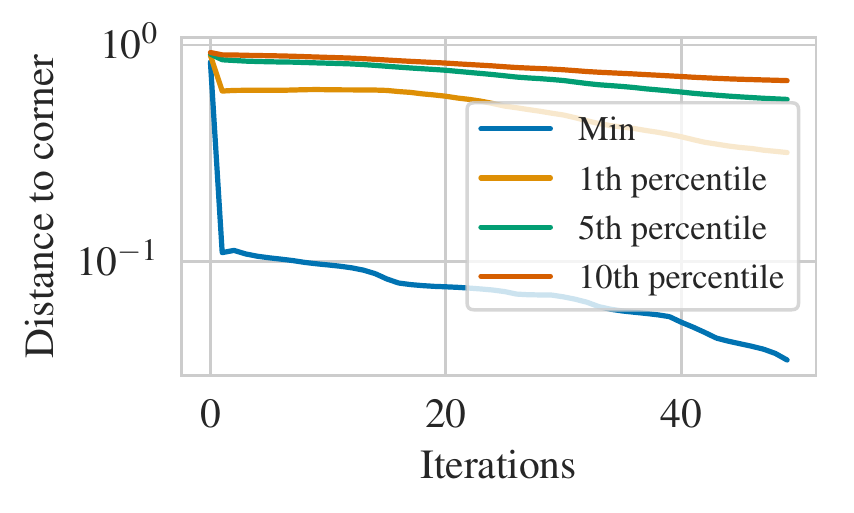}
        \caption{Distribution of distance to potentially degenerate points.}
        \label{fig:distance_to_corner}      
    \end{subfigure}%
    \hfill %
    \caption{Evaluation on the Medium-sized network with SoftPlus activation function }
\end{figure}

\begin{table}
\floatbox[{\capbeside\thisfloatsetup{floatwidth=sidefil,capbesideposition={right,top},capbesidewidth=.32\linewidth}}]{table}
{}%
{
}
\end{table}

Verification of these networks is relaxed to a stage-wise convex problem (Appendix~\ref{sec:certification-of-neural-networks}).
We compare three strategies for solving this relaxation: 
(i) NonConvex, our nonconvex reformulation using \callSimplePsiMin{} augmented with momentum and backtracking linesearch,
(ii) DeepVerify~\citep{Dvijotham2018} (DV) that performs Lagrangian relaxation on the bound computation problem,
(iii) a direct encoding of the relaxation into CVXPY~\citep{cvxpy}, with  SCS~\citep{Odonoghue2016} and ECOS~\citep{domahidi2013ecos} backends\footnote{We also ran tests on an internal primal-dual hybrid gradient implementation. It was not remotely competitive (failing to converge after 100,000 iterations on trialed instances) so we did not include it in the results.}.
We terminate (i) after 50 iteration or when the relative duality gap is less than $10^{-2}$, (ii)
after 500 iterations or when its dual value is larger than the final value of NonConvex (NC) (details in Appendix~\ref{app:implementation-details}).


Table~\ref{tab:benchmark} shows that, compared with the specialized first-order method DV, our method is faster by a factor between 3 and 50 depending on the network architecture, and always produces tighter bounds. As the two methods solve problems that have the same optimal value, we hypothesize that the discrepancy is because the Lagrangian relaxation of DV contains an inner-maximization problem that makes its objective extremely non-smooth, slowing convergence. 

In most problems, DV reaches the imposed iterations limit before convergence. This is quantified in Table~\ref{tab:convergence_prop} where we show that beyond the tiny network, DV does not reach a small enough dual gap to achieve early stopping. On the other hand, we observe that for NC, the scale of the network does not significantly impact the required number of iterations. Figure~\ref{fig:opt_trace} shows an example of the evolution of the computed bound, where we can see that the objective of DV plateaus, while NC converges in few iterations. Since the time per iteration for both methods is roughly the same, our runtime is lower.

After a single iteration, the duality gap achieved by our method is considerably smaller. The variables of DV exist on an unbounded feasible domain and appropriate initial values are therefore difficult to estimate, leading to large initial duality gap. Our method does not suffer from this problem, as all our variables are constrained between 0 and 1, and we can therefore initialize them all to 0.5, which empirically gives us good performance.

\begin{wraptable}[13]{r}{.5\linewidth}
    \vspace{-10pt}
    \caption{\small Proportion of bound computations on CIFAR-10 where the algorithm converges within the iteration budget, and average number of iterations.}\label{tab:convergence_prop}
    \small\addtolength{\tabcolsep}{-5pt}
    \begin{tabular}{lcc@{\hskip 0.2in}cc}
            \toprule
            & \multicolumn{2}{l}{\bf Early stopping \%} & \multicolumn{2}{c}{\bf Avg iteration count}\\
            \cmidrule(r{1em}){2-3} \cmidrule{4-5}

            & DV & NC & DV & NC\\ 
            \midrule
            Tiny ReLU & 37\% & 73\% & 384 & 18 \\
            Small ReLU & 0\% & 97\% & 500 & 9 \\
            Medium ReLU & 63\% & 100\%& 284 & 5 \\
            \midrule
            Tiny SoftPlus & 14 \% & 100\%& 467 & 4\\
            Small SoftPlus & 0 \% & 100\%& 500 & 7\\
            Medium SoftPlus & 0 \% & 59\%& 500& 25\\
            \bottomrule
   \end{tabular}
\end{wraptable}

\noindent {\bf Nondegeneracy in practice.} In Section~\ref{sec:FO-global-in-nondegenerate-case}, we described a simple version of our algorithm under the assumption that the algorithm does not enter a degenerate region.
In the context of Neural Network verification, due to the structure of the problem, the only possibility for a small gap between $\eta_i - \mu_i$ is at the boundary of the feasible domain of the convex hull relaxation of activation.
Even points close to the corner are not necessarily degenerate as they also need to satisfy a condition on the gradients.
Throughout optimization, we measure $\frac{\min\{z_i - l_i, u_i - z_i \}}{u_i - l_i}$ where $l_i$ and $u_i$ are lower and upper bounds on $z_i$ (corresponding to the corners), as shown in Figure~\ref{fig:distance_to_corner}.
We can observe that this value is strictly positive for all $i$ which means we are not entering the degenerate region. This explains why, for these problems, \callSimplePsiMin{} was able to converge to good solutions.

\paragraph{Conclusion:} We have developed a novel algorithm for a class of stage-wise convex optimization problems. Our experiments showed that our algorithm is efficient at solving standard relaxations of neural network verification problems. We believe that these results will generalize to stronger relaxations \citep{anderson2020strong}, as well as other stage-wise convex problems such as those arising in optimal control and generalized isotonic regression. 

\newpage

\section*{Broader Impact}
Our work leads to new scalable algorithms for verifying properties of neural networks and solve certain kinds of structured regression problems. On the positive side, these can have an impact in terms of better methods to evaluate the reliability and trustworthiness of state of the art deep learning systems, thereby catching any unseen failure modes and preventing undesirable consequences of deep learning models. On the negative sign, the algorithms are agnostic to the type of properties being verified and may facilitate abuses by allowing attackers to verify that their attacks can reliabily induces specific failure modes in a deep learning model. Further, any applications of these techniques is reliant on carefully designing desirable specifications or properties of a deep learning model - if this is not done carefully, even systems that are verifiable with these algorithms may exhibit undesirable behavior (arising from bias in the data or the specification). 

\begin{ack}

We thank Miles Lubin for establishing the connections between the authors and helpful feedback on the paper.
We'd also like to thank Ross Anderson, Christian Tjandraatmadja, and Juan Pablo Vielma for helpful discussions.
\end{ack}

\bibliography{specific-bio.bib,generic-bio.bib}

\newpage

\appendix

\section{Examples of optimization problems with this structure}\label{appendix:examples-with-this-structure}

\subsection{Linear quadratic control and extensions}\label{sec:linear-quadratic-control}

Linear quadratic control problems \cite{lewis2012optimal} take the form:
\begin{subequations}\label{LQR}
\begin{flalign}
\minimize{x}{&\frac{1}{2}\sum_{t=1}^{M} x(t)^T Q(t) x(t) + u(t)^T R(t) u(t)} \\
x(t+1) &= 
A(t) x(t) + B(t) u(t) \\
x(0) &= x_{\text{initial}}
\end{flalign}
\end{subequations}
where $A(t), B(t) \in \R^{n \times n}$ are matrices, $Q(t) , R(t) \in \R^{n \times n}$ are symmetric positive definite matrices, and $x(t) \in \R^{n}$ represents the system state, $u(t) \in \R^{n}$ the input, the initial system state is $x_{\text{initial}} \in \R^{n}$, and the positive integer $M$ is the number of time steps.
This problem can be solved by dynamic programming using $O(M n^3)$. 

Another approach \cite{lasdon1967conjugate} is to reformulate by eliminating the $x$ variables by forward propagation, thereby rewriting the problem as
\begin{flalign}\label{eq:eliminate-u}
\minimize{u}{~h(u)}.
\end{flalign}
The gradient of \eqref{eq:eliminate-u} can be computed by backpropagation which takes time proportional to the total number of non-zeros in $A(t)$, $B(t)$, $Q(t)$, and $R(t)$. One can therefore solve this problem using gradient descent and due to the lower iteration cost, potentially find an approximate minimizer faster than using dynamic programming. The function $h(u)$ is a convex quadratic, which implies gradient descent finds the global minimizer. It is worth noting that applying our nonconvex reformulation to \eqref{LQR} by letting $s \gets u$ and $z \gets x$ yields $h(u) = \psi_0(u, \theta)$. Therefore, for linear quadratic control our approach and the approach of \citet{lasdon1967conjugate} are essentially equivalent.

However, one benefit of our approach is that we can tackle a wider range of problems than these classical methods. For example, we can support more complex dynamics where $x(t+1)$ is wedged between a convex and concave function. 

\subsection{Verification of neural networks robustness to adversarial attacks}\label{sec:certification-of-neural-networks}

To provide guarantees on the behaviour of neural networks, there has been a surge of interest in verifying that the output classification of a trained model remains stable when the input is slightly perturbed (adversarially) \citep{bunel2018unified,Dvijotham2018, Wong2018provable}.
In particular, consider an input $s_0$ with label $c^{*}$. We wish to show there exists no adversarial example close to $s_0$ such that the network outputs $c \neq c^{*}$.
Define $S$ as the restriction of the input domain over which we want to perform verification. 
In the context of robustness to adversarial attacks, this would typically correspond to $S = \left\{s \ |\  \| s - s_0 \|_\infty \leq \epsilon \right\}$.
The set of feasible activation values for a feedfoward neural network for any input $s \in S$ satisfy,
\begin{subequations}
\label{eq:feedforward-network}
\begin{flalign}
s &\in S \\
z_{1:j(1)} &= \sigma ( W_0 s + b_0) \\
z_{j(k)+1:j(k+1)} &= \sigma ( W_k z_{j(k-1)+1:j(k)} + b_k ), \quad \forall k \in \{ 1, \dots, K - 1  \},
\end{flalign}
\end{subequations}
where $k$ represents layers, $j(0), \dots, j(K)$ partitions the vector $z$ such that $z_{j(k-1)+1:j(k)}$ is the activation values for layer $k$, $W_k$ is the weight matrix for the $k$th layer, and $\sigma$ represents the activation function (e.g., ReLU).
By splitting the matrices $W_k$ into a sequence of vectors $w_i$ and the vectors $b_k$ into a sequence of numbers $h_i$ this can be rewritten in the form
\begin{subequations}
\label{eq:general-network}
\begin{flalign}
s &\in S \\
z_i &= \sigma ( [s, z_{1:i-1}] \cdot w_i + h_i ) \quad \forall i \in \{1, \dots, n \}.
\end{flalign}
\end{subequations}
To make the definitions precise,
$w_i = [\mathbf{0}, [W(k)]_{i-j(k-1)}]$ where $k$ is the unique solution to \mbox{$j(k-1) + 1 \le i \le j(k)$}, and $[W(k)]_{i-j(k-1) }$ denotes the $i-j(k-1)$th row of $W(k)$; $h_i = [b_k]_{i-j(k-1) }$. Note that \eqref{eq:general-network} is more general than \eqref{eq:feedforward-network} as it could capture more than just feedforward networks.

We now describe a procedure to verify the network.
Let $C$ be the set of possible output classes from the network and $v_{c}$ a weight vector for each class $c \in C$.
Typically neural networks classify an example according to the rule
$$
\argmax_{c \in C} {v_{c} \cdot z}.
$$
Usually, $v_c$ is a sparse vector with zeros in all entries except those corresponding to the last layer of the network.

Therefore to verify that the network will output class $c^{*}$ for all inputs in $S$ it suffices to solve
$$
\minimize{z}{~(v_{c^{*}} - v_{c}) \cdot z} 
\quad\text{subject to}\quad \eqref{eq:general-network},
$$
for each $c \in C \setminus \{ c^{*} \}$. If the minimum value of each of these subproblems is positive then the network is robust to adversarial perturbations.

Unfortunately, this problem is intractable as the feasible region given by \eqref{eq:general-network} is nonconvex, and moreover the problem is in general NP-hard \cite{katz2017reluplex}. However, this does not preclude the possibility of verifying the neural network by forming a convex relaxation of \eqref{eq:general-network}. To form this convex relaxation of \eqref{eq:general-network}, we need lower and upper bounds on the possible values for each value of $[s, z_{1:i-1}] \cdot w_i$. 
These bounds can be obtained either by optimization over the partially constructed problem or by simple bound propagation~\citep{Gowal2019}.
Let us denote these bounds by $l_i$ and $u_i$. In the case where $\sigma$ is a ReLU we define the convex relaxation in the form of \eqref{eq:convex-relaxation} with
\begin{flalign*}
f(s,z) &= (v_{c} - v_{c'}) \cdot z
\end{flalign*}
and for all $i \in \{1, \dots, n\}$,
\begin{flalign*}
\mu_i ( s, z_{1:i-1} ) &= \sigma ( [s, z_{1:i-1}] \cdot w_i + h_i ) = \max\{ [s, z_{1:i-1}] \cdot w_i + h_i, 0 \} \\
\eta_i( s, z_{1:i-1} ) &=
\begin{cases}
\frac{u_i}{u_i - l_i} ([s, z_{1:i-1}] \cdot w_i + h_i - l_i) & 0 \in [l_i, u_i] \\
[s, z_{1:i-1}] \cdot w_i + h_i & l_i \ge 0 \\
0 & u_i \le 0
\end{cases}
\label{eq:relu_cvx_hull}
\end{flalign*}
where $\mu_i$ and $\eta_i$ are depicted in Figure~\ref{compare:original-and-convex}.
Since ReLU is a convex function, this feasible region is convex. Definition of the constraints corresponding to the convex hull relaxation of different type of non-linearities have been previously published in the literature~\citep{salman2019convex,bunel2020lagrangian}.

Due to the way that the lower and upper bounds are constructed, this convex relaxation satisfies Assumption~\ref{assume:inductive}.
In particular, if we form the bounds by optimizing over the partially constructed problem, i.e.,
\begin{flalign*}
l_j &= h_j + \minimize{(s,z) \in S \times \R^{n}}{~[s, z_{1:j-1}] \cdot w_j} \quad \text{s.t.} \quad \mu_i ( s, z_{1:i-1} ) \le z_i \le \eta_i ( s, z_{1:i-1} ) \quad \forall i \in \{ 1, \dots, j - 1 \} \\
u_j &= h_j + \maximize{(s,z) \in S \times \R^{n}}{~[s, z_{1:j-1}] \cdot w_j} \quad  \text{s.t.} \quad \mu_i ( s, z_{1:i-1} ) \le z_i \le \eta_i ( s, z_{1:i-1} ) \quad \forall i \in \{ 1, \dots, j - 1 \}
\end{flalign*}
then we can see that given we have a feasible solution $[s, z_{1:j-1}]$ to $\mu_i ( s, z_{1:i-1} ) \le z_i \le \eta_i ( s, z_{1:i-1} ) ~ \forall i \in \{ 1, \dots, j - 1 \}$ then $[s, z_{1:j-1}] \cdot w_j + h_j \in [l_i, u_i]$ which implies that for $0 \not\in [l_i, u_i]$ that $\eta_j(s,z_{1:j-1})-\mu_j(s,z_{1:j-1}) \ge 0$ by definition and if $0 \in [l_i, u_i]$ then
\begin{flalign*}
\eta_j(s,z_{1:j-1})-\mu_j(s,z_{1:j-1}) &= \frac{u_i}{u_i - l_i} ([s, z_{1:i-1}] \cdot w_i - l_i) - \max\{ [s, z_{1:i-1}] \cdot w_i, 0 \} \\
&\ge \min \left\{ \frac{u_i}{u_i - l_i} (u_i -  l_i),
\frac{u_i}{u_i - l_i} (l_i - l_i)
\right\}  \\
&= 0
\end{flalign*}
as required to establish Assumption~\ref{assume:inductive} (intuition for this can be given by contrasting Figure~\ref{fig:infeasibility-example} with Figure~\ref{compare:original-and-convex}).
By a similar argument, simple bound propagation~\citep{Gowal2019} to compute $l_i$ and $u_i$ will also guarantee Assumption~\ref{assume:inductive} holds.
In general, if a relaxation is constructed in an inductive fashion Assumption~\ref{assume:inductive} tends to naturally  hold.

\begin{figure}[tbhp]
    \centering
    \includegraphics[height=100pt]{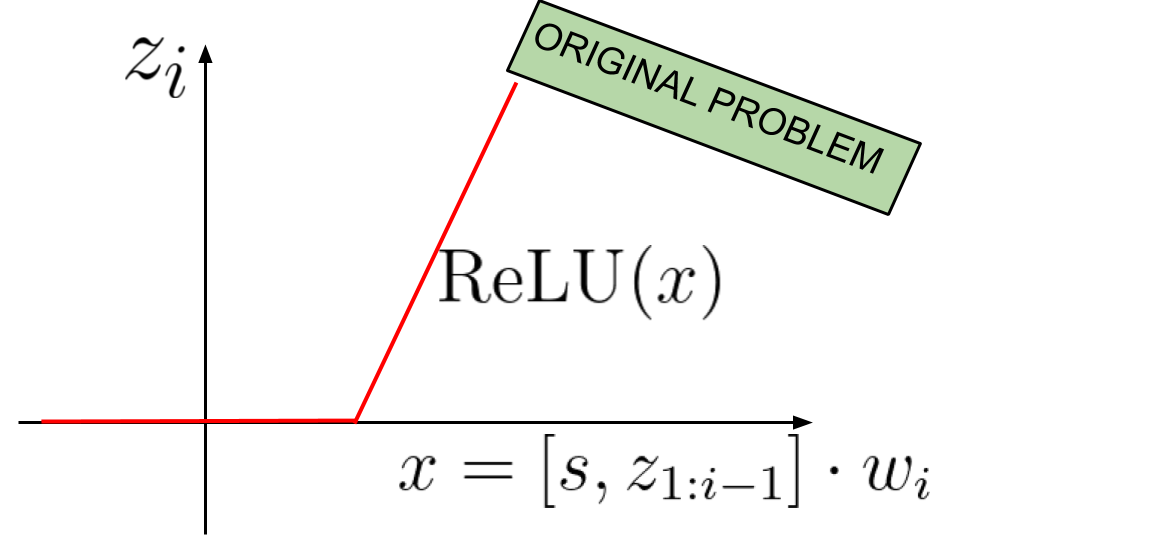}\includegraphics[height=100pt]{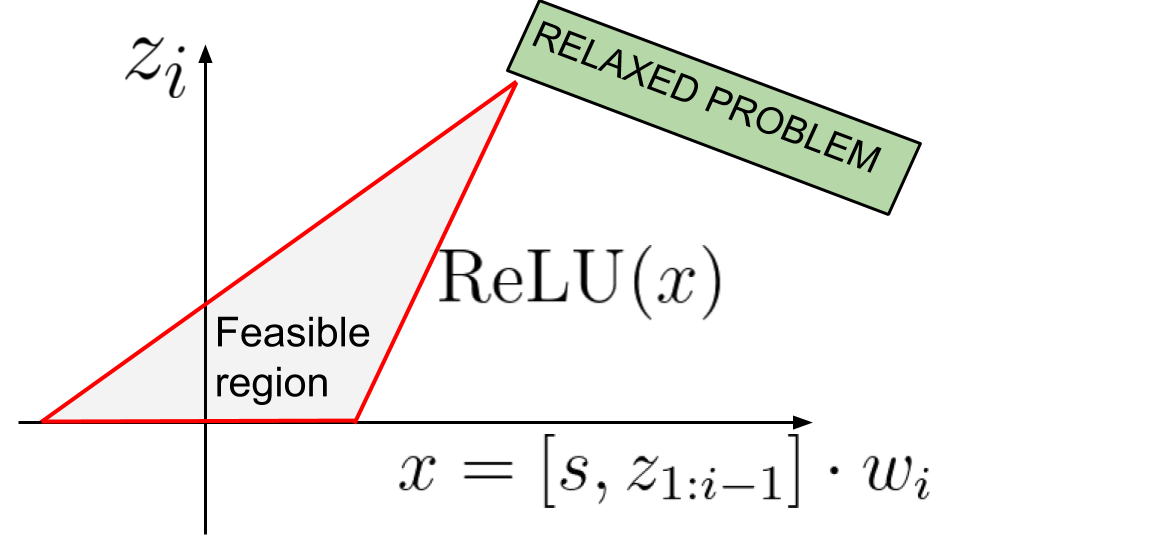}
    \caption{Comparison of original feasible region with convex relaxation when $\sigma$ is ReLU.}
    \label{compare:original-and-convex}
\end{figure}

There exists tighter linear programming bounds for neural network verification that require an (implicit) exponential number of inequalities \cite{anderson2020strong}. 
Ignoring smoothness issues, these also can be cast into our framework.
Using our approach to solve these linear programs is an interesting avenue for future work.
These types of relaxations also have applications in reinforcement learning \cite{ryu2019caql}.

\paragraph{Related work on Deep Network verification} The convex relaxation was proposed by Ehlers [Ehlers, ATVA 2017] but the high computational cost of solving it with off-the-shelf LP solvers was prohibitive for large instances. A large number of papers such as IBP \cite{gowal2018effectiveness}, DeepPoly \cite{singh2019abstract}, Crown \cite{zhang2018efficient}, Neurify \cite{wang2018efficient}, LP-relaxed-dual \cite{Wong2018provable} focused on looser
relaxations to allow for fast, closed form solutions of the bound computation problem scaling to larger networks. In parallel, work was done to reformulate the optimization problem to allow the use of better algorithms: DeepVerify
\cite{dvijotham2018dual} introduced an unconstrained dual reformulation of the non-convex problem and showed equivalence with the convex relaxation. Proximal \cite{bunel2020lagrangian} performed lagrangian decomposition and used proximal methods to solve the problem faster. The application of our method to network certification follows this research direction of speeding up the computation of network bounds without compromising on tightness.

\subsection{Finite horizon Markov decision processes with cost function controlled by an adversary}\label{finite-horizon-markov-processes}

Consider a finite horizon Markov Decision Process with uncertain rewards. In particular, at each time stage $t \in \{1, \dots, T \}$ we take an action $a \in A$ and move from state $k \in K$ to a state $k' \in K$ with probability $P_{a, k,k',t}$ and earn reward $R_{a,k,t}$. Further suppose the rewards are uncertain but we know $R \in \mathcal{R}$ where $\mathcal{R}$ is a convex set. We wish to know what the optimal policy is given we start in state $1$. This can be written as the following optimization problem
\begin{subequations}\label{eq:MDP}
\begin{flalign}
\min_{R \in \mathcal{R}, v} & \quad v_{1,1} \\
\max_{a \in A} \sum_{k' \in K} P_{a,k,k',t} v_{k',t+1} + R_{a,k,t} &\le v_{k,t}
\end{flalign}
\end{subequations}
where $v_{k,T} = 0$. 
\citet{mcmahan2003planning} develops a specialized algorithm for this problem. This algorithm is used for robots playing laser tag \cite{rosencrantz2003locating}.

We can also cast this problem in our framework. Suppose we have an upper bound on the rewards $R_{a,k,t}^{\max}$ ($R_{a,k,t}^{\max} \ge R_{a,k,t}$ for all $R \in \mathcal{R}$) then applying standard dynamic programming we can compute $v_{k,t}^{\max}$ allowing us to rewrite \eqref{eq:MDP} in the form,
\begin{flalign*}
\min_{R \in \mathcal{R}, v} & \quad v_{1,1} \\
\max_{a \in A} \sum_{k' \in K} P_{a,k,k',t} v_{k',t+1} + R_{a,k,t} &\le v_{k,t} \le v^{\max}_{k,t}.
\end{flalign*}
For this problem one can show Assumption~ \ref{assume:bounded-set} and \ref{assume:inductive} are satisfied. The only issue is that $\max_{a \in A} \sum_{k' \in S} P_{a,k,k',t} v_{k',t+1} + R_{a,k,t}$ is a nonsmooth function. Although, as we mention in Remark~\ref{remark:smoothing-issue}, this problem is likely surmountable.

\subsection{Generalized Isotonic regression}\label{generalized-isotonic-regression}


Classic isotonic regression considers the following problem
\begin{flalign}
\min_z \sum_{i=1}^n (z_i - y_i)^2 \quad \text{s.t.} \quad z_j \le z_i \quad \forall ~ j \prec i.\label{eq:isotonic-regression}
\end{flalign}
This problem has applications in genetics \cite{gamarnik2019sparse,luss2012efficient}, psychology \cite{kruskal1964multidimensional}, and biology \cite{obozinski2008consistent}.
When $\prec$ is a total ordering then there exists efficient algorithms that use only linear time \cite{best1990active}. However, for the general case developing efficient algorithms is an area of active research \cite{luss2012efficient}. 

Our approach offers a new way of solving these problems.
In particular, if we add the mild assumption that $z_i$ is bounded in $[l,u]$ then \eqref{eq:isotonic-regression} reduces to
\begin{flalign}
\min_z \sum_{i=1}^n (z_i - y_i)^2 \quad \text{s.t.} \quad \max\{ z_j, l \} \le z_i \le u \quad \forall ~ j \prec i.
\end{flalign}
Note that if $\prec$ is a partial ordering then there exists a total ordering that is consistent with this partial ordering. We can represent such a total ordering by a permutation $\pi : \{1, \dots, n \} \rightarrow \{1, \dots, n\}$ where $i \preceq j \Rightarrow \pi(i) \le \pi(j)$. Therefore, without loss of generality assume that that $i \preceq j \Rightarrow i \le j$. This allows us to write the problem in the form of \eqref{eq:convex-relaxation}, as
$$
\min_z \sum_{i=1}^n (z_i - y_i)^2 \quad \text{s.t.} \quad \mu_i(z_{1:i-1}) \le z_i \le \eta_i(z_{1:i-1}), \quad \forall i \in \{1, \dots, n \},
$$
where $\mu_i(z_{1:i-1}) = \max\left\{ l, \maximum{j : j \prec i}{~z_j} \right\}$ and $\eta_i(z_{1:i-1}) = u$. Given
$\mu_i(z_{1:i-1}) \le \eta_i(z_{1:i-1})$ for all $i < j$ then $z_i \le u$ which implies that
$\eta_j(z_{1:j-1}) \le u$, i.e., Assumption~\ref{assume:inductive} holds.

\section{Sketch of lower bounds}
\label{appendix:lower-bound}

Here we briefly sketch how to use to show that solving \eqref{eq:hard-first-order-method-instance} requires at least $n - 1$ iterations of saddle point methods.
We only provide a sketch since very similar results are already known \cite{zhang2019lower,ouyang2019lower}.
Before reading this section we recommend reading a standard reference on lower bounds (e.g., \cite[Section 3.5]{bubeck2014convex}, \cite[Section~2.1.2]{nesterov2013introductory}) and a standard reference on saddle point methods (e.g., \cite[Section 5.2]{bubeck2014convex}).

We can reformulate \eqref{eq:hard-first-order-method-instance} as a saddle point problem as follows
\begin{flalign*}
\minimize{x \in X}{\quad \maximize{y \in Y}{-e_n^T x + y^T A x}}
\end{flalign*}
where $e_n$ is a vector with a one in the $n$th entry and all other values zero, $X := \{ x \in \R^{n} : x \ge -1, x_1 \in [0, 1] \}$, $Y := \{ y \in \R^{n-1} : y \ge 0 \}$ and
$$
A = \begin{pmatrix}
1 & -1 & & & \\
& 1 & -1 &  & \\
& & \dots & & \\
&  & & 1 & -1 \\
\end{pmatrix}.
$$
After appropriate reindexing the iterates $\{ (x^t, y^t) \}_{t=0}^{\infty}$, saddle point algorithms such as primal-dual hybrid gradient \cite{chambolle2011first} and mirror-prox \cite{nemirovski2004prox} satisfy
\begin{flalign*}
\hat{x}_{t+1} &\in x_t + \text{span}(A^T y_0, \dots, A^T y_t) \\
x_{t+1} &= \Pi_X(\hat{x}_{t+1}) \\
\hat{y}_{t+1} &\in y_t + \text{span}(A x_0, \dots,  A x_t) \\
y_{t+1} &= \Pi_Y(\hat{y}_{t+1})
\end{flalign*}
where $\Pi_{X}$ and $\Pi_{Y}$ projects onto the set $X$ and $Y$ respectively.

Define $Z^t_X = \{ x \in \R^n : x_i = 0, \forall i \in \{1 ,\dots, n - t \} \}$ and $Z^t_Y = \{ y \in \R^{n-1} : y_i = 0, \forall i \in \{1 ,\dots, n - t - 1 \} \}$.
Now, if $t \in \{0, \dots n -2 \}$, $x^0, \dots, x^t \in Z^t_X$ and $y^0, \dots, y^t \in Z^t_Y$ then 
\begin{itemize}
    \item $\hat{x}^{t+1} \in x^0 +  \text{span}(A^T y^0, \dots, A^T y^t) \subseteq x^t + Z_X^{t+1} = Z_X^{t+1} \Rightarrow x^{t+1} \in Z_X^{t+1}$. 
    \item $\hat{y}^{t+1} \in y^0 +  \text{span}(A x^0, \dots, A x^t) \subseteq y^t + Z_Y^{t+1} = Z_Y^{t+1} \Rightarrow y^{t+1} \in Z_Y^{t+1}$
\end{itemize}
where given a set $S$ and a vector $s'$ we define the addition of them by $S + s' := \{ s + s' : s \in S \}$.

Therefore if $x^0 = \mathbf{0} \in Z_{X}^0$ and $y^0 = \mathbf{0} \in Z_{X}^0$ then by induction $x_1^t = 0$ for $t < n$.

\section{Proof results from Section~\ref{sec:FO-global-in-nondegenerate-case}}

We will use the following fact throughout the proofs.

\begin{fact}
Let $(s,\theta) \in S \times [0,1]^{n}$ and $z = \Forward{s,\theta}$.
If Assumption~\ref{assume:inductive} holds then $\mu_i \le \eta_i$.
\end{fact}
As discussed in Section~\ref{sec:global}, this follows immediately from Theorem~\ref{lem:equivalence}.

\subsection{Proof of Lemma~\ref{lem:progress-relationship-to-gap}}\label{sec:proof-of:lem:progress-relationship-to-gap}

Lemma~\ref{eq:gd-progress-lemma} is a standard result on the relationship between progress made by a gradient step and the gradient of a function.

\begin{lemma}\label{eq:gd-progress-lemma}
Suppose $s \in S$.
If Assumption~\ref{assume:bounded-set} holds then
$\sup_{\hat{s} \in S} \grad_s \psi_0 \cdot (s - \hat{s})  \le D_s \sqrt{2 L \delta_L}$.
\end{lemma}

\begin{proof}
Note that for any $\hat{s} \in S$,
$$
\delta_L(s,\theta) \ge -\minimize{\alpha \in \R}{~ \alpha \grad_s \psi_0 \cdot (\hat{s} - s) + \frac{\sm}{2} \| \hat{s} - s \|_2^2 \alpha^2} = \frac{(\grad_s \psi_0 \cdot (s - \hat{s}))^2}{2 \sm \| \hat{s} - s \|_2^2} 
$$
where the equality uses that $\alpha = \frac{\grad_s \psi_0 \cdot (s - \hat{s})}{\sm \| \hat{s} - s \|_2^2}$ minimizes the quadratic. Rearranging yields, 
$$
\grad_s \psi_0 \cdot (s - \hat{s}) \le \| \hat{s} - s \|_2 \sqrt{2 L \delta_L(s,\theta)}.
$$
\end{proof}

\begin{lemma}\label{lem:preliminary-bound}
Suppose Assumption~\ref{assume:bounded-set}  holds, then
\begin{flalign*}
\sup_{(\hat{s},\hat{z}) \in \S \times \Z} \grad_{s,z} \lag \cdot (s - \hat{s}, z - \hat{z}) \le D_s \sqrt{2 L \delta_L} +  D_s \| \grad_s \lag - \grad_s \psi \|_2 + D_z \| \grad_z \lag \|_2.
\end{flalign*}
\end{lemma}

\begin{proof}
Moreover,
\begin{flalign*}
&\sup_{(\hat{s},\hat{z}) \in \S \times \Z} \grad_{s,z} \lag \cdot (s - \hat{s}, z - \hat{z}) \\
&=  \sup_{\hat{s} \in S} \grad_s \lag \cdot (s - \hat{s}) + \sup_{\hat{z} \in Z} \grad_z \lag \cdot (z - \hat{z}) &  \\
&\le \sup_{\hat{s} \in S} \grad \psi_s \cdot (s - \hat{s}) + \sup_{\hat{s} \in S} (\grad_s \lag - \grad_s \psi) \cdot (s - \hat{s}) + \sup_{z \in Z} \grad_z \lag \cdot (z - \hat{z}) \\
&\le D_s \sqrt{2 L \delta_L} +  D_s \| \grad_s \lag - \grad_s \psi \|_2 + D_z \| \grad_z \lag \|_2 & \\
\end{flalign*}
where the final inequality uses $\sup_{\hat{s} \in S} \grad \psi_s \cdot (s - \hat{s}) \le D_s \sqrt{2 L \delta_L} $ (Lemma~\ref{eq:gd-progress-lemma}) and Assumption~\ref{assume:bounded-set}.
\end{proof}

Display~\eqref{eq:valid-duality-gap} establishes that $\Delta(s,\theta)$ is a valid duality gap.
Moreover, Lemma~\ref{lem:preliminary-bound} shows that to provide an upper bound on $\Delta(s,\theta)$ it will suffice to upper bound
$$
\sum_{i=1}^{n} ( y\ind{i} z_{i} - y\ind{i}^{+} \mu\ind{i} + y\ind{i}^{-} \eta\ind{i}) + D_s \sqrt{2 L \delta_L} +  D_s \| \grad_s \lag - \grad_s \psi \|_2 + D_z \| \grad_z \lag \|_2.
$$
Lemma~\ref{lem:gradient-bounds} is our first step towards bounding these quantities. Before proceeding with Lemma~\ref{lem:gradient-bounds} we prove a fact we will find useful.

\begin{fact}\label{fact:gamma-alpha-useful}
Let $\gamma, t, \alpha, \beta \in \R$ then
\begin{flalign*}
(\gamma^{+} \beta - \gamma^{-} \alpha) - \gamma (t \alpha + (1 - t) \beta ) =  ( \beta - \alpha) (\gamma^{+} t + \gamma^{-} (1 - t) ).
\end{flalign*}
\end{fact}

\begin{proof}
Observe that
\begin{flalign*}
\gamma^{+} (\beta - (t \alpha + (1 - t) \beta) ) &= \gamma^{+} t (\beta - \alpha) \\
\gamma^{-} (-\alpha + (t \alpha + (1 - t) \beta) ) &= \gamma^{-} (1 - t) (\beta - \alpha),
\end{flalign*}
adding the two expressions together gives the result.
\end{proof}

\begin{lemma}\label{lem:gradient-bounds}
Suppose Assumption~\ref{assume:V-W} holds.
Let $y_i =  \frac{\partial \psi_{i}}{\partial z_{i}}$ and $r_i = \theta_i y^{+}_i +  (1 - \theta_i) y^{-}_i$, then
\begin{subequations}
\begin{flalign}
\| \grad_s \lag - \grad_s \psi_0 \|_2 &\le \| \grad_s \mu - \grad_s \eta \|_2 \| r \|_2   \label{eq:bound-L-s-i} \\
\| \grad_z \lag \|_2 &\le  \| \grad_z \mu - \grad_z \eta \|_2 \| r \|_2 \label{eq:bound-L-z-i} \\
\sum_{i=1}^{n} ( y\ind{i} z_{i} - y\ind{i}^{+} \mu\ind{i} + y\ind{i}^{-} \eta\ind{i}) &\le \| \eta - \mu \|_2 \| r \|_2.
\label{eq:bound-comp-term}
\end{flalign}
\end{subequations}
\end{lemma}

\begin{proof}
Consider the expansion of \eqref{eq:grad-psi-s} and \eqref{eq:grad-psi-z} using $\psi_n = f$:
\begin{flalign*}
\grad_s \psi_{i} &= 
\grad_s f + \sum_{j = i + 1}^{n} \frac{\partial \psi_{j}}{\partial z_{j}} \left( \theta_{j} \grad_s \eta_{j} + ( 1 - \theta_{j}) \grad_{s} \mu_{j} \right) \\
\frac{\partial \psi_{i}}{\partial z_{k}} &= 
\frac{\partial f}{\partial z_{k}} + \sum_{j = i + 1}^{n} \frac{\partial \psi_{j}}{\partial z_{j}} \left( \theta_{j} \frac{\partial \eta_{j}}{\partial z_k} + ( 1 - \theta_{j}) \frac{\partial \mu_{j}}{\partial z_k} \right)
\end{flalign*}
for each $i \in \{1, \dots, n\}$ and $k \in \{ 1, \dots, i \}$.
Setting $k = i$ gives
\begin{subequations}\label{eq:grad-psi-expanded-and-simplified}
\begin{flalign}
\label{eq:grad-psi-expanded-and-simplified:s-i}
\grad_s \psi_{0} &= 
\grad_s f + \sum_{j = 1}^{n} \frac{\partial \psi_{j}}{\partial z_{j}} \left( \theta_{j} \grad_s \eta_{j} + ( 1 - \theta_{j}) \grad_{s} \mu_{j} \right) \\
\label{eq:grad-psi-expanded-and-simplified:z-i}
\frac{\partial \psi_{i}}{\partial z_{i}} &= 
\frac{\partial f}{\partial z_{i}} + \sum_{j = i + 1}^{n} \frac{\partial \psi_{j}}{\partial z_{j}} \left( \theta_{j} \frac{\partial \eta_{j}}{\partial z_i} + ( 1 - \theta_{j}) \frac{\partial \mu_{j}}{\partial z_i} \right).
\end{flalign}
\end{subequations}
Contrast \eqref{eq:grad-psi-expanded-and-simplified} with
\begin{subequations}\label{eq:grad-lag-value}
\begin{flalign}
\label{eq:grad-lag-value:s-i}
\grad_s \lag &= \grad_s f + \sum_{j=1}^{n} \left( y\ind{j}^{+} \grad_s \mu_j - y\ind{j}^{-} \grad_s \eta\ind{j} \right) \\
\frac{\partial \lag}{\partial z_i} &= \frac{\partial f}{\partial z_i} - y_{i} + \sum_{j=i+1}^{n} \left( y\ind{j}^{+} \frac{\partial \mu\ind{j}}{\partial z_i} - y\ind{j}^{-} \frac{\partial \eta\ind{j}}{\partial z_i} \right).
\label{eq:grad-lag-value:z-i}
\end{flalign}
\end{subequations}
One can see \eqref{eq:grad-psi-expanded-and-simplified} and \eqref{eq:grad-lag-value} share a very similar structure which we will exploit. In particular, 
\begin{flalign*}
\grad_s \lag - \grad_s \psi_0 &= \sum_{i=1}^{n} \left( y\ind{j}^{+} \grad_s \mu\ind{j} - y\ind{j}^{-} \grad_s \eta\ind{j} -  y_j \left( \theta_j \grad_s \eta_{j} + ( 1 - \theta_{j}) \grad_s \mu_{j} \right) \right) \\
&=\sum_{j=1}^{n} \left(  \grad_s \mu_{j} - \grad_s \eta_{j} \right) \left( \theta_j y^{+}_j +  (1 - \theta_j) y^{-}_j \right).
\end{flalign*}
where the first equality subtracts \eqref{eq:grad-psi-expanded-and-simplified:s-i} from \eqref{eq:grad-lag-value:s-i}, and the second equality uses Fact~\ref{fact:gamma-alpha-useful}.
We conclude \eqref{eq:bound-L-s-i} holds. 
Similarly,
\begin{flalign*}
\frac{\partial \lag}{\partial z_i} &= \sum_{j=i+1}^{n} \left( y\ind{j}^{+} \frac{\partial \mu\ind{j}}{\partial z_i} - y\ind{j}^{-} \frac{\partial \eta\ind{j}}{\partial z_i} -  y_j \left( \theta_j \frac{\partial \eta_{j}}{\partial z_i} + ( 1 - \theta_{j}) \frac{\partial \mu_{j}}{\partial z_i} \right) \right) \\
&=\sum_{j=i+1}^{n} \left(  \frac{\partial \mu_{j}}{\partial z_i} - \frac{\partial \eta_{j}}{\partial z_i} \right) \left( \theta_j y^{+}_j +  (1 - \theta_j) y^{-}_j \right). 
\end{flalign*}
where the first equality substitutes $y_i = \frac{\partial \psi_{i}}{\partial z_{i}}$ into \eqref{eq:grad-lag-value:z-i} and then subtracts \eqref{eq:grad-psi-expanded-and-simplified:z-i} from \eqref{eq:grad-lag-value:z-i}, and the second equality uses Fact~\ref{fact:gamma-alpha-useful}.
We conclude \eqref{eq:bound-L-z-i} holds. Finally,
$z_i - \mu\ind{i} = (1 - \theta\ind{i}) \mu\ind{i} + \theta\ind{i} \eta\ind{i} - \mu\ind{i} = \theta_i (\eta_i - \mu_i)$ and $\eta\ind{i} - z_i = (1 - \theta_i) (\eta_i - \mu_i)$ which implies
$$
\sum_{i=1}^{n} y\ind{i}^{+} (z\ind{i} - \mu\ind{i}) + y\ind{i}^{-} (\eta\ind{i} - z\ind{i}) = \sum_{i=1}^{n} (\eta\ind{i} - \mu\ind{i}) (y\ind{i}^{+} \theta_i + y\ind{i}^{-} (1 - \theta_i)) = \sum_{i=1}^{n} r\ind{i} (\eta\ind{i} - \mu\ind{i}),
$$
establishing \eqref{eq:bound-comp-term}.
\end{proof}

\begin{lemma}\label{lem:generic-gap-bound}
Suppose Assumption~\ref{assume:bounded-set} holds.
Let $y_i =  \frac{\partial \psi_{i}}{\partial z_{i}}$ and $r_i = \theta_i y^{+}_i +  (1 - \theta_i) y^{-}_i$, then
$\Delta(s,\theta) \le
D_s \sqrt{2 L \delta_L} +
\left( \| \eta - \mu \|_2  + D_s \| \grad_s \mu - \grad_s \eta \|_2 + D_z \| \grad_z \mu - \grad_z \eta \|_2  \right) \| r \|_2$.
\end{lemma}

\begin{proof}
Note that by \eqref{eq:valid-duality-gap}, Lemma~\ref{lem:preliminary-bound} and \ref{lem:gradient-bounds},
\begin{flalign*}
\Delta(s, \theta) 
&\le \sum_{i=1}^{n} ( y\ind{i} z_{i} - y\ind{i}^{+} \mu\ind{i} + y\ind{i}^{-} \eta\ind{i}) + D_s \sqrt{2 L \delta_L} +  D_s \| \grad_s \lag - \grad_s \psi \|_2 + D_z \| \grad_z \lag \|_2  \\
&\le \left( \| \eta - \mu \|_2 \| r \|_2 + D_s \left(\sqrt{2 L \delta_L} + \| \grad_s \mu - \grad_s \eta \|_2 \| r \|_2 \right) + D_z \| \grad_z \mu - \grad_z \eta \|_2 \| r \|_2 \right).
\end{flalign*}
\end{proof}
We will find Lemma~\ref{lem:generic-gap-bound} useful later in Section~\ref{app:escaping-the-basin-of-a-local-minimizer}.

Next, define
$$
\gain{i,\sm} := - \minimize{\theta_{i} + d_i \in [0,1]}{~\frac{\partial \psi_0}{\partial \theta_i} d_i + \frac{\sm}{2} d_i^2}
$$
which represent the guaranteed reduction from a gradient step, contributed by $\theta_i$, assuming the function $\psi_0$ is $\sm$-smooth.

While Lemma~\ref{lem:generic-gap-bound} represents useful progress in bounding $\Delta(s,\theta)$.
We would like our final bound on $\Delta(s,\theta)$ to depend only on $\delta_L$ and problem constants.  Lemma~\ref{lem:r-i-bound} allows us to do that.

\begin{lemma}\label{lem:r-i-bound}
Suppose $\eta_i - \mu_i > 0$. Let $y_i =  \frac{\partial \psi_{i}}{\partial z_{i}}$ and $r_i = \theta_i y^{+}_i +  (1 - \theta_i) y^{-}_i$, then
\begin{flalign}\label{bound:rj}
r_i \le \frac{\max\left\{\sqrt{2 \sm \gain{i,\sm}},  2 \gain{i,\sm}\right\}}{\eta_i - \mu_i} .
\end{flalign}
\end{lemma}

\begin{proof}

Suppose $\frac{\partial \psi_0}{\partial \theta_i} \ge 0$ and $\theta_i \le \frac{1}{L} \frac{\partial \psi_0}{\partial \theta_i}$ then
$$
\gain{i,\sm} \ge \theta_i \left( \frac{\partial \psi_0}{\partial \theta_i}  - \frac{L \theta_i}{2} \right) \ge \frac{\theta_i}{2} \frac{\partial \psi_0}{\partial \theta_i}.
$$
If $\frac{\partial \psi_0}{\partial \theta_i} \ge 0$ and $\theta_i \ge \frac{1}{L} \frac{\partial \psi_0}{\partial \theta_i}$ then $\gain{i,\sm} \ge \frac{1}{2 L} \left( \frac{\partial \psi_0}{\partial \theta_i} \right)^2$. Therefore, if $\frac{\partial \psi_0}{\partial \theta_i} \ge 0$ then
\begin{flalign}
\gain{i,\sm} \ge \frac{1}{2} \min\left\{ \frac{1}{L} \left( \frac{\partial \psi_0}{\partial \theta_i} \right)^2, \theta_i  \frac{\partial \psi_0}{\partial \theta_i} \right\}\label{eq:t-i-L-bound-psi-positive}
\end{flalign}
which implies that
\begin{flalign}\label{eq:bound-pos-r-i}
(\eta_i - \mu_i) \frac{\partial \psi_i}{\partial z_i} \theta_i &= \frac{\partial \psi_0}{\partial \theta_i} \theta_i \le \max\left\{ \theta_i \sqrt{2 \sm \gain{i,\sm}}, 2 \gain{i,\sm}  \right\}
\end{flalign}
where the equality uses \eqref{eq:grad-psi-theta} and the inequality rearranges \eqref{eq:t-i-L-bound-psi-positive}.
By the same argument, if $\frac{\partial \psi_0}{\partial \theta_i} \le 0$ then
\begin{flalign}\label{eq:bound-neg-r-i}
-(\eta_i - \mu_i) \frac{\partial \psi_i}{\partial z_i} (1 - \theta_i) \le \max\left\{ (1 - \theta_i) \sqrt{2 \sm \gain{i,\sm}}, 2 \gain{i,\sm}  \right\}.
\end{flalign}
By \eqref{eq:bound-pos-r-i} and \eqref{eq:bound-neg-r-i} we deduce \eqref{bound:rj}.
\end{proof}

\begin{proof}[Proof of Lemma~\ref{lem:progress-relationship-to-gap}]
Observe that
\begin{flalign}\label{eq:gi-max-bound}
\gamma^2 \| r \|_2^2 \le \sum_{i=1}^{n} \max\left\{ 2 L \gain{i,L}, 4 \gain{i,L}^2 \right\} \le \sum_{i=1}^{n} 2 L \gain{i,L} + 4 \gain{i,L}^2 \le 2 L \delta_L + 4 \delta_L^2 \le 4 L \delta_L
\end{flalign}
where the first inequality uses \eqref{bound:rj}, the second inequality uses that $\sum_{i=1}^{n} \gain{i,L} \le \delta_L$ and last inequality uses the assumption that $\delta_L \le L / 2$. Combining equation~\eqref{eq:gi-max-bound}, Lemma~\ref{lem:generic-gap-bound} and Assumption~\ref{assume:V-W} gives the result.
\end{proof}

\subsection{Proof of Theorem~\ref{thm:simple-psi-minimization}}\label{sec:simple-psi-minimization-proof}

Rather than directly prove the Theorem~\ref{thm:simple-psi-minimization}, we first prove Lemma~\ref{lem:generic-gradient-descent-proof} which is a generic statement on the convergence of algorithms to minimizers. 
We will find Lemma~\ref{lem:generic-gradient-descent-proof} useful later in Section~\ref{app:escaping-the-basin-of-a-local-minimizer}.

\begin{lemma}\label{lem:generic-gradient-descent-proof}
Let $\zeta_1, \zeta_2, \zeta_3 \in (0,\infty)$.
Consider a sequence $(s^k, \theta^k)_{k=0}^{\infty}$ satisfying
\begin{flalign} \label{eq:generic-gap-to-progress-relation}
\left(f(s^{k}, \theta^{k}) - \fStar \right)^{\zeta_1+1} \le \zeta_2 ( f(s^{k}, \theta^{k}) - f(s^{k+1}, \theta^{k+1}) )
\end{flalign}
for all $\frac{f(s^{1}, \theta^{1}) - \fStar}{f(s^{k}, \theta^{k}) - f(s^{k+1}, \theta^{k+1})} \le \zeta_3$.
Then for $K > \zeta_3$
$$
f(s^{k}, \theta^{k}) - \fStar \le \left( \frac{\zeta_2}{K - \zeta_3} \right)^{1/\zeta_1}.
$$
\end{lemma}

\begin{proof}
Define,
$f(s^{k}, \theta^{k}) - \fStar = v^k$.
First consider the case that $\frac{f(s^{k}, \theta^{k}) - f(s^{k+1}, \theta^{k+1})}{f(s^{1}, \theta^{1}) - \fStar} \le \zeta_3$, then by \eqref{eq:generic-gap-to-progress-relation} and
$v^k - v^{k+1} = f(s^{k},\theta^k) - f(s^{k+1},\theta^{k+1})$ we deduce
$$
\frac{( v^k )^{\zeta_1+1}}{\zeta_2} \le v^k - v^{k+1} \Rightarrow v^{k+1} \le v^{k} \left(1 - \frac{ \left(v^k \right)^{\zeta_1}}{\zeta_2}  \right).
$$
Dividing both sides by $v^{k+1} (v^k)^{\zeta_1}$ and using that $v^{k+1} \le v^{k}$ yields
\begin{flalign}
\label{eq:vk-inverse-bound}
\frac{1}{(v^{k})^{\zeta_1}} \le \frac{v^k}{v^{k+1}} \left( \frac{1}{(v^{k})^{\zeta_1-1}} - \frac{1}{\zeta_2}  \right) \le \frac{1}{(v^{k+1})^{\zeta_1}} - \frac{1}{\zeta_2}.
\end{flalign}
Furthermore, if $\frac{f(s^{k}, \theta^{k}) - f(s^{k+1}, \theta^{k+1})}{f(s^{1}, \theta^{1}) - \fStar} \ge \zeta_3$ then 
\begin{flalign}
v^{k+1} \le v^{k} - \frac{f(s^{1}, \theta^{1}) - \fStar}{\zeta_3}
\end{flalign}
and this can happen at most $\zeta_3$ times. Therefore if $K > \zeta_3$ then
$$
\frac{1}{(v^{K})^{\zeta_1}} \ge \frac{1}{(v^{K})^{\zeta_1}} - \frac{1}{(v^{1})^{\zeta_1}} = \sum_{k=1}^{K-1}\frac{1}{(v^{k+1})^{\zeta_1}} - \frac{1}{(v^{k})^{\zeta_1}} \ge \frac{K - \zeta_3}{\zeta_2}
$$
where the first inequality uses that $v_1 \ge 0$, and second inequality uses \eqref{eq:vk-inverse-bound}.
Rearranging gives the result.
\end{proof}

\begin{proof}[Proof of Theorem~\ref{thm:simple-psi-minimization}] 
Define \begin{flalign*}
\zeta_1 &= 1 \\
\zeta_2 &= L \left( D_s \sqrt{2} + 2 \frac{c}{\gamma} \right)^2 \\
\zeta_3 &= \frac{2\Delta(s^{1}, \theta^{1})}{L}.
\end{flalign*}

Lemma~\ref{lem:progress-relationship-to-gap} shows that if $\delta_L(s^k, \theta^{k}) \le \Delta(s^{1}, \theta^{1}) / \zeta_3$ then
$\Delta(s^{k}, \theta^{k})^2 \le \zeta_2 \delta_L(s^k, \theta^{k})$.
Furthermore, $\delta_L(s^k, \theta^{k}) \le  \psi_0(s^{k}, \theta^{k}) - \psi_0(s^{k+1}, \theta^{k+1}) = f(s^k,z^k) - f(s^{k+1},z^{k+1})$ by line~\ref{algline:optstep} of Algorithm~\ref{alg:PGD} and the definition of $\psi_0$ respectively. Combining these two inequalities yields
$(f(s^k,z^k) -f_{*})^2 \le \zeta_2 (f(s^k,z^k) - f(s^{k+1},z^{k+1}))$. Applying Lemma~\ref{lem:generic-gradient-descent-proof} to the latter inequality yields the result.
\end{proof}

\section{Escaping the basin of a local minimizer}\label{app:escaping-the-basin-of-a-local-minimizer}

For this section, we consider an alternative  set to $\mathcal{K}_{\gamma}$:
\begin{flalign*}
\mathcal{C}(s, \theta, q) := \Bigg\{  i : \theta_i  \left( \frac{\partial \psi_i}{\partial z_i} \right)^{+} + (1 - \theta_i) \left( \frac{\partial \psi_i}{\partial z_i}\right)^{-} > 2 q (
\eta_i - \mu_i) \Bigg\}
\end{flalign*}
this set identifies the indices where the degeneracy could cause a convergence issue. We then modify \callSimplePsiMin{} (yielding \callSafePsiMin{}) to detect when the set $C(s,\theta,q)$ is empty and take appropriate action, i.e., call \callFixDeg{}.
\callFixDeg{} extends \callEscapeExactLocalMin{} to allow us to escape from the basin of a local minimizer.
To see this, note that if $(s,\theta)$ is fixed then setting $q$ sufficiently large will cause \callFixDeg{$s,\theta,q$} to reduce to \callEscapeExactLocalMin{$s,\theta
$}. However, the value of $q$ needed to achieve this could be arbitrarily large.

\begin{algorithm}[H]
\caption{Local search algorithm that will find the global minimizer of $\psi_0$.}\label{alg:safe-PGD}
\begin{algorithmic}[1]
\Function{\stateSafePsiMin}{$s^{1}, \theta^{1}, L, q^{1}, \epsilon$}
\For{$k = 1, \dots, \infty$}
\State \textit{Take corrective action if degeneracy is an issue:}
\If{$\mathcal{C}(s^{k}, \theta^{k}, q^{k}) \neq \emptyset$}\label{line:if-C-is-empty}
\State $(s^{k}, \hat{\theta}^{k}, \textbf{status}) \gets \callFixDeg{s^{k}, \theta^{k}, q^{k}}$
\If{$\textbf{status} = \textsc{FAILURE}$}
\State $q^{k+1} \gets 10 q^{k}$
\EndIf
\Else \label{line:else-C-is-empty}
\State $\hat{\theta}^{k} \gets \theta^{k}$
\EndIf \label{line:end-if-C-is-empty}
\State \textit{Termination checks:}
\If{$\Delta(s^{k}, \hat{\theta}^{k}) \le \epsilon$}
\State \Return $(s^{k}, \theta^{k})$
\EndIf
\State \textit{Reduce the function at least as much as PGD would:}
\State $(s^{k+1}, \theta^{k+1}) \in \{(s,\theta) : \psi_0(s, \theta) \le \psi_0(s^{k}, \hat{\theta}^{k}) - \delta_{\sm}(s^{k}, \hat{\theta}^k) \}$
\EndFor
\EndFunction
\end{algorithmic}
\end{algorithm}

\begin{algorithm}[H]
\caption{Algorithm for fixing convergence issues in degenerate case}\label{alg:fix-deg}
\begin{algorithmic}[1]
\Function{\stateFixDeg}{$s, \theta, q$}
\State $z = \Forward{s, \theta}$
\State $\hat{\theta} \gets \text{copy}(\theta)$
\State \emph{The minimum reduction in $\psi_0$ if $q \ge Q$:}
\State $v \gets 0$
\State
\For{$i = n, \dots, 1$}
\State \emph{Approximately compute $\frac{\partial \psi_i(s, z_{1:i}, \hat{\theta}_{i+1:n})}{\partial z_i}$:}
\State $g_i \gets \frac{\partial f}{\partial z_{i}} + \sum_{j = i + 1}^{n} g_{j} \left( \hat{\theta}_{j} \frac{\partial \eta_{j}}{\partial z_i} + ( 1 - \hat{\theta}_{j}) \frac{\partial \mu_{j}}{\partial z_i} \right)$
\State \emph{Estimate the distance that $z$ has moved:}
\State $\omega_i \gets \sum_{j = i+1}^{n} \abs{\theta_j - \hat{\theta}_j} (\eta_j - \mu_j)$
\State $p_i \gets g_i^{+} \theta_i + g_{i}^{-} (1 - \theta_i)$ \label{line:define-p-i}
\If{$p_i > 2 q (\omega_i + \eta_i - \mu_i)$}
\State \emph{Fix degeneracy in index $i$:}
\State $
\hat{\theta}_{i} \gets \begin{cases}
0 & g_i > 0 \\
1 & g_i < 0
\end{cases}
$ \label{line:choose-theta-hat}
\State $v \gets v + \frac{1}{2} p_i (\eta_i - \mu_i)$
\EndIf
\EndFor
\State 
\State $\omega_0 \gets \sum_{j = 1}^{n} \abs{\theta_j - \hat{\theta}_j} (\eta_j - \mu_j)$ \label{line:omega-0-def}
\If{$\psi_0(s, \hat{\theta}) > \psi_0(s, \theta) - v$} \label{line:sufficient-progress}
\State \Return $(s, \theta, \textsc{FAILURE})$
\ElsIf{$\abs{g_i - \frac{\partial \psi_i(s,\hat{z}_{1:i},\hat{\theta}_{i+1:n})}{\partial \hat{z}_i}} \le q \omega_0, \forall i$} \label{line:q-omega-1-bound}
\State \Return $(s, \hat{\theta}, \textsc{SUCCESS})$
\Else
\State \Return $(s, \hat{\theta}, \textsc{FAILURE})$
\EndIf \label{line:end-algorithm}
\EndFunction
\end{algorithmic}
\end{algorithm}

\begin{remark}
With careful implementation the cost of running \callFixDeg{} is the same as one backpropagation.
\end{remark}

\begin{remark}
If $\mathcal{C}(s,\theta,q) = \emptyset$ then $(s, \hat{\theta}, \textbf{status}) \gets \callFixDeg{s, \theta, q}$ satisfies $\hat{\theta} = \theta$ and $\textbf{status} = \textsc{SUCCESS}$. In other words, removing Line~\ref{line:if-C-is-empty} and Line~\ref{line:else-C-is-empty}-\ref{line:end-if-C-is-empty} of \callSafePsiMin{} does not change the behaviour of the Algorithm (although it may create unnecessary computation).
\end{remark}

This section introduces two new assumptions (Assumption~\ref{assume:g-i-predict} and \ref{assume:eta-mu-smooth-bound}). 
We defer justifying these assumptions to Section~\ref{sec:justifying-assumptions} where we show that Assumptions~\ref{assume:smooth}, \ref{assume:bounded-set}, and \ref{assume:inductive} imply that these introduced assumptions hold.

\begin{assumption}\label{assume:g-i-predict}
Let $z \gets \Forward{s, \theta}$, $$
g_i = \frac{\partial f}{\partial z_{i}} + \sum_{j = i + 1}^{n} g_{j} \left( \hat{\theta}_{j} \frac{\partial \eta_{j}}{\partial z_i} + ( 1 - \hat{\theta}_{j}) \frac{\partial \mu_{j}}{\partial z_i} \right),
$$
and $\hat{z} \gets \Forward{s, \hat{\theta}}$. Then there exists a constant $Q > 0$ such that
$$
\abs{g_i - \frac{\partial \psi_i(s, \hat{z}_{1:i}, \hat{\theta}_{i+1:n})}{\partial \hat{z}_i}} \le Q \sum_{i=1}^{n} \abs{\theta_i - \hat{\theta}_i} \abs{\eta_i - \mu_i}.
$$
\end{assumption}

\begin{lemma}\label{lem:fix-deg-succeeds}
Suppose Assumption~\ref{assume:inductive} and~\ref{assume:g-i-predict} holds.
If $q \ge Q$ then $(s, \hat{\theta}, \textbf{status}) \gets \callFixDeg{s, \theta, q}$ has $\textbf{status} = \textsc{SUCCESS}$.
\end{lemma}

\begin{proof}
By Lines~\ref{line:sufficient-progress} to \ref{line:end-algorithm} of \callFixDeg{} we can see for $\textbf{status} = \textsc{SUCCESS}$ we need both
\begin{flalign}\label{proof:bound-g-i-minus-partial-psi-i}
\abs{g_i - \frac{\partial \psi_i(s,\hat{z}_{1:i},\hat{\theta}_{i+1:n})}{\partial \hat{z}_i}} \le q \omega_0, \forall i
\end{flalign}
and
\begin{flalign}\label{proof:bound-progress-by-v}
\psi_0(s, \hat{\theta}) \le \psi_0(s, \theta) - v.
\end{flalign}
We establish each of these in turn.
First observe that
\eqref{proof:bound-g-i-minus-partial-psi-i}holds immediately by Assumption~\ref{assume:g-i-predict} and definition of $\omega_0$ in line~\ref{line:omega-0-def} of \callFixDeg{}.

Next we show \eqref{proof:bound-progress-by-v} by bounding each of the right hand side terms in the equality
\begin{flalign}\label{eq:sum-of-psi}
\psi_0(s, \hat{\theta}) - \psi_0(s,\theta) =  \sum_{i=1}^{n} \psi_0(s, \theta_{1:i-1}, \hat{\theta}_{i:n}) - \psi_0(s, \theta_{1:i}, \hat{\theta}_{i+1:n}).
\end{flalign}

Recall the definition of $p_i$, $\omega_i$, and $\hat{\theta}_i$ from \callFixDeg{}. If $p_i > 2 q (\omega_i + \eta_i - \mu_i)$ then
\begin{flalign}
&\psi_0(s, \theta_{1:i-1}, \hat{\theta}_{i:n}) - \psi_0(s, \theta_{1:i}, \hat{\theta}_{i+1:n}) \notag \\
&= \int_{\theta_i}^{\hat{\theta}_i} \frac{\partial \psi_0(s, \theta_{1:i-1}, \gamma, \hat{\theta}_{i+1:n})}{\partial \gamma} \partial \gamma \notag \\
&= \int_{\theta_i}^{\hat{\theta}_i} 
\frac{\partial \psi_{i}(s, z_{1:i-1}, t, \hat{\theta}_{i+1:n})}{\partial t}(\eta_i - \mu_i) \partial t
& \text{by \eqref{eq:grad-psi-theta}}
\notag \\
&= \int_{\theta_i}^{\hat{\theta}_i} g_i (\eta_i - \mu_i) \partial t
+ \int_{\theta_i}^{\hat{\theta}_i} 
\left( \frac{\partial \psi_{i}(s, z_{1:i-1}, t, \hat{\theta}_{i+1:n})}{\partial t} - g_i\right) (\eta_i - \mu_i) \partial t \notag \\
&\le g_i (\eta_i - \mu_i) \int_{\theta_i}^{\hat{\theta}_i}  \partial t
+ Q \left( \sum_{j = i+1}^{n} \abs{\theta_j - \hat{\theta}_j} (\eta_j - \mu_j) \right) (\eta_i - \mu_i) \abs{ \int_{\theta_i}^{\hat{\theta}_i}  \partial t } 
& \text{by Assumption~\ref{assume:g-i-predict}}
\notag \\
&= g_i (\eta_i - \mu_i) \int_{\theta_i}^{\hat{\theta}_i}  \partial t
+ Q \omega_i (\eta_i - \mu_i) \abs{ \int_{\theta_i}^{\hat{\theta}_i}  \partial t } 
& \text{by definition of $\omega_i$}
\notag \\
&=  g_i (\eta_i - \mu_i) (\hat{\theta}_i - \theta_i) +  Q \omega_i  (\eta_i - \mu_i) \abs{\hat{\theta}_i -  \theta_i} \notag \\
&\le  g_i (\eta_i - \mu_i) (\hat{\theta}_i - \theta_i) +  q \omega_i  (\eta_i - \mu_i) \abs{\hat{\theta}_i -  \theta_i}
& \text{by $Q \le q$}
\notag \\
&\le  g_i (\eta_i - \mu_i) (\hat{\theta}_i - \theta_i) +  q \omega_i  (\eta_i - \mu_i)
& \text{as $\theta_i, \hat{\theta}_i \in [0,1]$}
\notag \\
&= (\eta_i - \mu_i) (q \omega_i - p_i) &\text{by definition of $p_i$ and $\hat{\theta}_i$}
\notag \\
&\le - \frac{p_i (\eta_i - \mu_i)}{2},
\label{eq:delta-psi-p-i-large}
\end{flalign}
where the last inequality uses $p_i > 2 q \omega_i$.

If $p_i \le 2 q (\omega_i + \eta_i - \mu_i)$ then
\begin{flalign}
\label{eq:psi-delta-p-i-small}
\psi_0(s, \theta_{1:i-1}, \hat{\theta}_{i:n}) - \psi_0(s, \theta_{1:i}, \hat{\theta}_{i+1:n}) = 0
\end{flalign}
Therefore, by \eqref{eq:sum-of-psi}, \eqref{eq:delta-psi-p-i-large}, \eqref{eq:psi-delta-p-i-small}, and definition of $v$ we establish
\eqref{proof:bound-progress-by-v}.
\end{proof}

\begin{assumption}
\label{assume:eta-mu-smooth-bound}
Denote $z = \Forward{s, \theta}$ and $\hat{z} = \Forward{s,\hat{\theta}}$.
There exists a constant $P > 0$ such that
$$
\abs{ \eta_i - \mu_i - (\hat{\eta}_i - \hat{\mu}_i) } \le P  \sum_{i=1}^{n}  \abs{\theta_i - \hat{\theta}_i} \abs{\eta_i - \mu_i},
$$
$\forall s \in S, \forall \theta, \hat{\theta} \in [0,1]^{n}$, $\forall i \in \{1, \dots, n \}$ with $\hat{\eta}_i := \eta_i(s,\hat{z}_{1:i-1})$ and $\hat{\mu}_i := \mu(s,\hat{z}_{1:i-1})$.
\end{assumption}

\begin{lemma}\label{lem:bound-complementarity-after-fix-deg}
Suppose that Assumption~\ref{assume:inductive} and \ref{assume:eta-mu-smooth-bound} hold.
Let $(s, \hat{\theta}, \textbf{status}) \gets \callFixDeg{s, \theta, q}$, and $\hat{z} = \Forward{s, \hat{\theta}}$.
If $\textbf{status} = \textsc{SUCCESS}$ then
\begin{flalign}
\label{eq:bound-partial-psi}
\hat{\theta}_i \left( \frac{\partial \psi_i(s, \hat{z}_{1:i}, \hat{\theta}_{i+1:n})}{\partial \hat{z}_i} \right)^{+} + (1 - \hat{\theta}_i) \left( \frac{\partial \psi_i(s, \hat{z}_{1:i}, \hat{\theta}_{i+1:n})}{\partial \hat{z}_i} \right)^{-} \le q ( (3 + P) \omega_0 + \hat{\eta}_i - \hat{\mu}_i )
\end{flalign}
and $\omega_0 \le \sqrt{2 v / q}$.
\end{lemma}

\begin{proof}
Denote $\hat{\psi}_i := \psi_i(s,\hat{z}_{1:i},\hat{\theta}_{i+1:n})$.
From Line~\ref{line:q-omega-1-bound} of \callSafePsiMin{}
\begin{flalign}\label{eq:tolerance-on-g-i}
\abs{g_i - \frac{\partial \hat{\psi}_i}{\partial \hat{z}_i}} \le q \omega_0.
\end{flalign}
Therefore,
\begin{flalign}
\label{eq:bound-psi-hat-by-g-i-q-omega-0}
\hat{\theta}_i \left( \frac{\partial \hat{\psi}_i}{\partial \hat{z}_i} \right)^{+} + (1 - \hat{\theta}_i) \left( \ \frac{\partial \hat{\psi}_i}{\partial \hat{z}_i} \right)^{-} \le g_i^{+} \hat{\theta}_i + g_{i}^{-} ( 1 - \hat{\theta}_i ) + q \omega_0
\end{flalign}
where the inequality uses \eqref{eq:tolerance-on-g-i} and $\hat{\theta}_i \in [0,1]$.
From Line~\ref{line:define-p-i}-\ref{line:choose-theta-hat} if $p_i > 2 q (\omega_i + \eta_i - \mu_i)$ then
$$
g_i^{+} \hat{\theta}_i + g_{i}^{-} ( 1 - \hat{\theta}_i ) = 0 \Rightarrow \hat{\theta}_i \left(  \frac{\partial \hat{\psi}_i}{\partial \hat{z}_i} \right)^{+} + (1 - \hat{\theta}_i) \left(  \frac{\partial \hat{\psi}_i}{\partial \hat{z}_i} \right)^{-} \le q \omega_0
$$
where the implication uses  \eqref{eq:bound-psi-hat-by-g-i-q-omega-0}.
On the other hand, if 
$$
g_i^{+} \theta_i + g_{i}^{-} (1 - \theta_i) = p_i \le 2 q (\omega_i + \eta_i - \mu_i)
$$ 
then by $\hat{\theta}_i = \theta_i$ and \eqref{eq:bound-psi-hat-by-g-i-q-omega-0},
\begin{flalign}\label{eq:bound-partial-psi-simple}
\hat{\theta}_i \left( \frac{\partial \hat{\psi}_i}{\partial \hat{z}_i} \right)^{+} + (1 - \hat{\theta}_i) \left( \ \frac{\partial \hat{\psi}_i}{\partial \hat{z}_i} \right)^{-} \le p_i + q \omega_0 = q (3 \omega_0 + \eta_i - \mu_i).
\end{flalign}
Furthermore, by Assumption~\ref{assume:eta-mu-smooth-bound}
\begin{flalign*}
\abs{ \eta_i - \mu_i - (\hat{\eta}_i - \hat{\mu}_i) } \le P \omega_0 
\end{flalign*}
which combined with \eqref{eq:bound-partial-psi-simple} yields \eqref{eq:bound-partial-psi}.

It remains to show that $\omega_0 \le \sqrt{2 v / q}$.
Note that
\begin{flalign}
\label{eq:bound-eta-minus-mu-by-difference-omega-i}
\eta_i - \mu_i \ge \abs{\theta_i - \hat{\theta}_i} (\eta_i - \mu_i) = \omega_{i-1} - \omega_{i}
\end{flalign}
where the inequality uses $\theta_i, \hat{\theta}_i \in [0,1]$.
Define $\mathcal{I} := \{ i \in \{1, \dots, n \} : \theta_{i} \neq \hat{\theta}_i \}$.
Finally,
\begin{flalign*}
v &= \frac{1}{2} \sum_{i \in \mathcal{I}} (\eta_{i} - \mu_i) p_i 
& \text{by definition of $v$} \\
&\ge \frac{1}{2} \sum_{i \in \mathcal{I}} (\omega_{i-1} - \omega_{i})  p_i 
& \text{by \eqref{eq:bound-eta-minus-mu-by-difference-omega-i}} \\
&> q \sum_{i \in \mathcal{I}} (\omega_{i-1} - \omega_{i}) \omega_{i-1}
& \text{by $p_i > 2 q(\omega_i + \eta_i - \mu_i) > 2 \omega_i$ for $i \in \mathcal{I}$}\\
&= q \sum_{i = 1}^{n} (\omega_{i-1} - \omega_{i}) \omega_{i-1} 
& \text{by $\omega_{i-1} - \omega_{i}$ for $i \not\in \mathcal{I}$}\\
&= q \sum_{i=1}^{n} \sum_{j=i}^n (\omega_{i-1} - \omega_{i}) (\omega_{j-1} - \omega_{j})  \\
&\ge \frac{q (\omega_0 - \omega_n)^2}{2} & \text{by Fact~\ref{fact:delta-i-delta-j}} \\
&= \frac{q \omega_0^2}{2} & \text{since $\omega_n = 0$}.
\end{flalign*}
Rearranging this inequality gives the result.
\end{proof}

\begin{fact}\label{fact:delta-i-delta-j}
Let $\delta_1, \dots, \delta_n \in \R$ then
$$
\left(\sum_{i=1}^{n} \delta_i \right)^2 \le 2 \sum_{i=1}^{n} \sum_{j=i}^{n} \delta_i \delta_j.
$$
\end{fact}

\begin{proof}
It follows by
\begin{flalign*}
\left(\sum_{i=1}^{n} \delta_i \right)^2 =  \sum_{i=1}^{n} \sum_{j=1}^{n} \delta_i \delta_j &= \sum_{i=1}^{n} \delta_i^2 + \sum_{i=1}^{n} \sum_{j=i+1}^{n} \delta_i \delta_j + \sum_{i=1}^{n} \sum_{j=1}^{i-1} \delta_i \delta_j\\
 &= \sum_{i=1}^{n} \delta_i^2 + 2 \sum_{i=1}^{n} \sum_{j=i+1}^{n} \delta_i \delta_j \\
&= -\sum_{i=1}^{n} \delta_i^2 + 2 \sum_{i=1}^{n} \sum_{j=i}^{n} \delta_i \delta_j.
\end{flalign*}
\end{proof}

\begin{lemma}\label{lem:bound-r-j-advanced}
Suppose that Assumption~\ref{assume:inductive} and \ref{assume:eta-mu-smooth-bound} hold.
Let $(s, \hat{\theta}, \textbf{status}) \gets \callFixDeg{s, \theta, q}$, and $\hat{z} = \Forward{s, \hat{\theta}}$.
If $\textbf{status} = \textsc{SUCCESS}$ then for all $i \in \{1, \dots, n \}$,
$$
\hat{r}_i \le \sqrt{q} \left(  (3 + P) \sqrt{v} +  \max\left\{\sqrt[4]{2 \sm \hatGain{i,\sm}},  \sqrt{2 \hatGain{i,\sm}} \right\} \right).
$$
with $\hat{y}_i :=  \frac{\partial \psi_{i}(s, \hat{z}_{1:i},\hat{\theta}_{i+1:n})}{\partial \hat{z}_{i}}$, $\hat{r}_i := \hat{\theta}_i \hat{y}^{+}_i +  (1 - \hat{\theta}_i) \hat{y}^{-}_i$, and $\hatGain{i,\sm} :=  - \minimize{\hat{\theta}_{i} + d_i \in [0,1]}{~\frac{\partial \psi_{i}(s, \hat{z}_{1:i},\hat{\theta}_{i+1:n})}{\partial \hat{\theta}_i} d_i + \frac{\sm}{2} d_i^2}$.
\end{lemma}

\begin{proof}
Define 
\begin{flalign*}
\gamma_i &:= \begin{cases}
\frac{1}{\hat{\eta}_i - \hat{\mu}_i} & \hat{\eta}_i - \hat{\mu}_i \neq 0 \\
\infty & \text{otherwise}.
\end{cases} \\ 
a &:= \max\left\{\sqrt{2 \sm \hatGain{i,\sm}},  2 \hatGain{i,\sm}\right\} \\
b &:= q (3 + P) \omega_0 \\
c &:= q.
\end{flalign*}
By Lemma~\ref{lem:r-i-bound}, $\hat{r}_i \le a \gamma_i$
and by Lemma~\ref{lem:bound-complementarity-after-fix-deg}, $\hat{r}_i \le b + c / \gamma_i$.
Therefore,
\begin{flalign*}
\hat{r}_i &\le \min\{ a \gamma_i, b + c / \gamma_i \}.
\end{flalign*}
Maximizing the upper bound with for $\gamma_i \in [0,\infty]$ gives
$$
a \gamma_i = b + c / \gamma_i \Rightarrow a \gamma_i^2 - b \gamma_i - c = 0 \Rightarrow \gamma_i = \frac{b + \sqrt{b^2 + 4 a c}}{2 a}
$$
where the second implication uses the quadratic formula and $\gamma_i \ge 0$.
Plugging this back in gives
$$
\hat{r}_i \le \frac{b + \sqrt{b^2 + 4 a c}}{2} \Rightarrow \hat{r}_i \le b + \sqrt{a c}.
$$
Therefore, using $\omega_0 \le \sqrt{2 v / q}$ from Lemma~\ref{lem:bound-complementarity-after-fix-deg} and the definition of $a, b, c$ we get
$$
\hat{r}_i \le (3 + P) \sqrt{2 v q} + \sqrt{q \max\left\{\sqrt{2 \sm \hatGain{i,\sm}},  2 \hatGain{i,\sm}\right\} }.
$$
\end{proof}

\begin{lemma}\label{lem:Delta-s-theta-bound}
Suppose that Assumption~\ref{assume:bounded-set},  \ref{assume:inductive}, \ref{assume:V-W}, and \ref{assume:eta-mu-smooth-bound} holds.
Define $\totalGain := \delta_L(s, \hat{\theta}) + v$.
If 
\begin{flalign}\label{eq:assume-total-gain-is-less-than}
\totalGain \le \min\left\{ 2L, \frac{2 L}{n (3 + P)^4}, 
\frac{4 D_s^4 L^2}{c^4} \right\}
\end{flalign}
then
$$\Delta(s,\hat{\theta})^4 \le 96 c^4 q^2 L \totalGain.$$
\end{lemma}

\begin{proof}
First observe that
\begin{flalign}
\notag\| r \|_2^4 &\le q^2 \left( (3 + P)^4 n v^2 + 4 \delta_L(s, \hat{\theta})^2 + 2 L \delta_L(s, \hat{\theta}) \right) \\
\notag &\le q^2 \left( (3 + P)^4 n \totalGain^2 + 4 \totalGain^2 + 2 L \totalGain \right) \\
&\le 6 q^2 L \totalGain
\label{eq:bound-norm-r}
\end{flalign}
where the first inequality uses
Lemma~\ref{lem:bound-r-j-advanced}, the second inequality uses the definition of $\totalGain$, and the third inequality uses \eqref{eq:assume-total-gain-is-less-than}.

Next,
\begin{flalign*}
\Delta(s,\hat{\theta}) &\le D_s \sqrt{2 L \delta_L(s,\hat{\theta})} +
c \| r \|_2 \\
 &\le D_s \sqrt{2 L \totalGain} +
c 
\sqrt[4]{6 q^2 L \totalGain} \\
&\le 2 c
\sqrt[4]{6 q^2 L \totalGain}
\end{flalign*}
where the first inequality uses Lemma~\ref{lem:generic-gap-bound} and Assumption~\ref{assume:V-W}, the second inequality uses \eqref{eq:bound-norm-r}, and the third uses \eqref{eq:assume-total-gain-is-less-than}.
\end{proof}

\begin{theorem}
\label{thm:general-convergence-result-explicit}
Suppose that Assumption~\ref{assume:bounded-set}, \ref{assume:inductive}, \ref{assume:V-W}, \ref{assume:g-i-predict}, and \ref{assume:eta-mu-smooth-bound} hold. Define
\begin{flalign*}
\zeta_2 &= 96 c^4 q^2 L \\
\zeta_3 &= \frac{\Delta(s^{1}, \theta^{1})}{\min\left\{ 2L, \frac{2 L}{n (3 + B)^4}, 
\frac{4 D_s^4 L^2}{c^4} \right\}}
\end{flalign*}
where $L$ is the smoothness constant for $\psi_0$.
Then for $k > \zeta_3$ \callSafePsiMin{} satisfies
$$
\Delta(s^k, \theta^{k}) \le \left( \frac{\zeta_2}{ k - \zeta_3} \right)^{1/3}
$$
\end{theorem}

\begin{proof}
Follows by Lemma~\ref{lem:Delta-s-theta-bound} and Lemma~\ref{lem:generic-gradient-descent-proof} with $\zeta_1 = 3$.
\end{proof}

\subsection{Proof of Theorem~\ref{thm:general-convergence-result}}\label{app:proof-of-general-convergence-result}

\begin{proof}
Assumption~\ref{assume:smooth} and \ref{assume:bounded-set} imply that $f$, $\eta_i$, and $\mu_i$ are $\beta$-smooth and $B$-Lipschitz for some constant $\beta, B > 0$. Since $\mu_i$ and $\eta_i$ are Lipschitz Assumption~\ref{assume:V-W} holds. Lemma~\ref{lem:g-minus-psi-hat} implies
Assumption~\ref{assume:g-i-predict} holds. 
Corollary~\ref{lem:assumption-holds:eta-mu-smooth-bound} implies Assumption~\ref{assume:eta-mu-smooth-bound} holds. Note that Corollary~\ref{lem:assumption-holds:eta-mu-smooth-bound} and Lemma~\ref{lem:g-minus-psi-hat} appear in Section~\ref{sec:justifying-assumptions}

With Assumption \ref{assume:g-i-predict}, and \ref{assume:eta-mu-smooth-bound} established, the result holds by Theorem~\ref{thm:general-convergence-result-explicit}.
\end{proof}

\section{Justifying assumptions}\label{sec:justifying-assumptions}

The purpose of this section is to show that if Assumptions~\ref{assume:smooth} and \ref{assume:bounded-set} hold then Assumption~\ref{assume:g-i-predict} and \ref{assume:eta-mu-smooth-bound} hold.

\begin{definition}
A function $h : X \rightarrow \R$ is $L$-smooth with respect to $\| \cdot \|$ if $\| \grad h (x) - \grad h (x') \|_{*} \le L \| x - x' \|$ for all $x, x' \in X$.
\end{definition}

\begin{definition}
A function $h : X \rightarrow \R$ is $B$-Lipschitz with respect to $\| \cdot \|$ if $\abs{h(x) - h(x')} \le B \| x - x' \|$  for all $x, x' \in X$ .
\end{definition}

\begin{fact}
Suppose that $h : X \rightarrow \R$ differentiable and $B$-Lipschitz with respect to $\| \cdot \|$, then $\| \grad h(x) \|_{*} \le B$ for all $x \in X$. 
\end{fact}

\begin{fact}\label{fact:smooth-implies-beta-smooth}
Suppose that the function $h : X \rightarrow \R$ is smooth and $X$ is bounded. Then (for any given norm) there exists constant $B$ and $\beta$ such that $h$ is $B$-Lipschitz and $\beta$-smooth.
\end{fact}

\subsection{Proof Assumption~\ref{assume:eta-mu-smooth-bound} holds}

\begin{lemma}\label{lem:bound-z-minus-z-hat}
Suppose that $\eta_i - \mu_i$ is $B$-Lipschitz with respect to the $\ell_1$-norm for $B > 0$.
Let $z = \Forward{s, \theta}$ and $\hat{z} = \Forward{\hat{s},\hat{\theta}}$.
Then
$$
\| z - \hat{z} \|_1 \le B^{-1} (1 + B)^{n} \sum_{i=1}^{n}  \left( \abs{\theta_i - \hat{\theta}_i} \abs{\eta_i - \mu_i} + \| s - \hat{s} \|_1 \right)
$$
for all $\theta \in [0,1]^{n}$.
\end{lemma}

\begin{proof}
Denote $\eta_i(s, z_{1:i-1})$, $\mu_i(s, z_{1:i-1})$,  $\eta_i(\hat{s}, \hat{z}_{1:i-1})$,  $\mu_i(\hat{s}, \hat{z}_{1:i-1})$ by $\eta_i, \mu_i, \hat{\eta}_i$ and $\hat{\mu}_i$ respectively. Observe that
\begin{flalign}
\notag
\abs{z_i - \hat{z}_i} &= \abs{\theta_i \eta_i + (1 - \theta_i) \mu_i - (\hat{\theta}_i \hat{\eta}_i + (1 - \hat{\theta}_i) \hat{\mu}_i)} \\
\notag
&= \abs{(\theta_i - \hat{\theta}_i) (\eta_i - \mu_i) + (1 - \hat{\theta}_i) (\mu_i - \hat{\mu}_i) + \hat{\theta}_i (\eta_i - \hat{\eta}_i)} \\
\notag
&\le \abs{(\theta_i - \hat{\theta}_i) (\eta_i - \mu_i)} + \abs{\mu_i - \hat{\mu_i} + \eta_i - \hat{\eta}_i} \\
\label{eq:wrkwe132E}
&\le \abs{(\theta_i - \hat{\theta}_i) (\eta_i - \mu_i)} + B (\| s - \hat{s} \|_1 + \| z_{1:i-1} - \hat{z}_{1:i-1} \|_1).
\end{flalign}
Applying \eqref{eq:wrkwe132E}
for $i = 1$ we have
$\abs{z_1 - \hat{z}_1} \le \abs{(\theta_i - \hat{\theta}_i) (\eta_i - \mu_i)} + B \| s - \hat{s} \|_1$.
Suppose that,
\begin{flalign*}
\| z_{1:i} - \hat{z}_{1:i} \|_1 \le \sum_{j=1}^{i} (B+1)^{i-j} \left( B \| s - \hat{s} \|_1 + \abs{(\theta_j - \hat{\theta}_j) (\eta_j - \mu_j)} \right)
\end{flalign*}
then by \eqref{eq:wrkwe132E} we deduce
\begin{flalign*}
\| z_{1:i+1} - \hat{z}_{1:i+1} \|_1 \le  \sum_{j=1}^{i+1} (B+1)^{1+i-j} \left( B \| s - \hat{s} \|_1 + \abs{(\theta_j - \hat{\theta}_j) (\eta_j - \mu_j)}\right).
\end{flalign*}
By induction and the fact $\sum_{j=1}^{n} (1+B)^{n-j} \le B (B+1)^n$, the result holds.
\end{proof}

\begin{corollary}\label{lem:assumption-holds:eta-mu-smooth-bound}
Suppose that $\eta_i - \mu_i$ is $B$-Lipschitz with respect to the $\ell_1$-norm for $B > 0$.
Let $z = \Forward{s, \theta}$ and $\hat{z} = \Forward{\hat{s},\hat{\theta}}$.
Then
$$
\abs{\eta_i - \mu_i - (\eta_i(\hat{s},\hat{z}_{1:i-1}) - \mu_i(\hat{s},\hat{z}_{1:i-1}))} \le \beta \| s - \hat{s} \|_1 + \beta B^{-1} (1 + B)^{n} \sum_{i=1}^{n}  \left( \abs{\theta_i - \hat{\theta}_i} \abs{\eta_i - \mu_i} + \| s - \hat{s} \|_1 \right)
$$
for all $\theta \in [0,1]^{n}$.
\end{corollary}

\begin{proof}
Follows from the assumption $\eta_i - \mu_i$ is $\beta$-smooth and Lemma~\ref{lem:bound-z-minus-z-hat}.
\end{proof}

Therefore, if Assumptions~\ref{assume:smooth} and \ref{assume:bounded-set} hold then, by Fact~\ref{fact:smooth-implies-beta-smooth} and Corollary~\ref{lem:assumption-holds:eta-mu-smooth-bound}, we conclude Assumption~\ref{assume:eta-mu-smooth-bound} holds.

\subsection{Proof Assumption~\ref{assume:g-i-predict} holds}

\begin{fact}\label{fact:bound-a}
Let $a \in \R^{n}, b \in \R^{n+1}$.
Suppose
$$
a_i := b_{n+1} + \sum_{j=i+1}^{n} a_j b_j,
$$
then $\| a \|_1 \le (1 + \| b \|_{\infty})^{n}$.
\end{fact}

\begin{proof}
Note that
$$
\abs{a_i} \le \| b \|_{\infty} + \| a_{i+1:n} \|_1 \| b \|_{\infty} \Rightarrow \| a_{i:n} \|_1 \le \| b \|_{\infty} + \| a_{i+1:n} \|_1 (1 + \| b \|_{\infty}).
$$
Therefore, if
\begin{flalign}\label{eq:induction-hypotheis-norm-a}
\| a_{i:n} \|_1 \le \sum_{j=1}^{n-i} \| b \|_{\infty} (1 + \| b \|_{\infty})^{j}.
\end{flalign}
then
$$
\| a_{i-1:n} \|_1 \le \| b \|_{\infty} + \sum_{j=1}^{n-i} \| b \|_{\infty} (1 + \| b \|_{\infty})^{j+1} = \sum_{j=1}^{n-(i-1)} \| b \|_{\infty} (1 + \| b \|_{\infty})^{j}.
$$
Since \eqref{eq:induction-hypotheis-norm-a} holds for $i=n$ by induction \eqref{eq:induction-hypotheis-norm-a} holds for all $i$. By the bound on the sum of a geometric series the result holds.
\end{proof}

\begin{fact}\label{fact:a-minus-a-hat-bound}
Let $a, c \in \R^{n}, b \in \R^{n \times n}$.
Suppose
$$
a_i := c_{i} + \sum_{j=i+1}^{n} a_j b_{i,j}, \quad  \hat{a}_i := \hat{c}_{i} + \sum_{j=i+1}^{n} \hat{a}_j \hat{b}_{i,j},
$$
then
$$
\| a - \hat{a} \|_1 \le \left( \| c - \hat{c} \|_{\infty} + (1 + \| \hat{b} \|_{\infty})^{n} \| \hat{b} - b \|_{\infty} \right)
\frac{(1 + \| b \|_{\infty})^{n}}{\| b \|_{\infty}}.
$$
\end{fact}

\begin{proof}
Note that
\begin{flalign*}
\abs{a_i - \hat{a}_i} &= \abs{ c_{i} - \hat{c}_{i} + \sum_{j=i+1}^{n} (a_j b_{i,j} - \hat{a}_j \hat{b}_{i,j})} \\
&= \abs{c_{i} - \hat{c}_{i} +\sum_{j=i+1}^{n} ((a_j - \hat{a}_j) b_j - \hat{a}_j (\hat{b}_{i,j} - b_{i,j}))} \\
&\le \| a_{i+1:n} - \hat{a}_{i+1:n} \|_1 \| b \|_{\infty} + \| c - \hat{c} \|_{\infty} + \| \hat{a} \|_{1} \| \hat{b} - b \|_{\infty} \\
&\le \| a_{i+1:n} - \hat{a}_{i+1:n} \|_1 \| b \|_{\infty} + \| c - \hat{c} \|_{\infty} + (1 + \| \hat{b} \|_{\infty})^{n} \| \hat{b} - b \|_{\infty}
\end{flalign*}
where the last inequality uses Fact~\ref{fact:bound-a}. By induction,
\begin{flalign*}
\| a_{i:n} - \hat{a}_{i:n} \|_1 \le \left( \| c - \hat{c} \|_{\infty} + (1 + \| \hat{b} \|_{\infty})^{n} \| \hat{b} - b \|_{\infty} \right) \sum_{j=1}^{n-i} (1 + \| b \|_{\infty})^j.
\end{flalign*}
By the bound on the sum of a geometric series the result holds.
\end{proof}

\begin{lemma}
\label{lem:g-minus-psi-hat}
Suppose $f$, $\eta_i$, and $\mu_i$ are $\beta$-smooth functions with respect to the $\ell_1$-norm. Also, suppose $f$, $\eta_i$ and $\mu_i$ are 
$B$-Lipschitz with respect to the $\ell_1$-norm and $\hat{\theta}, \theta \in [0,1]^{n}$.
Let $z \gets \Forward{s, \theta}$, $$
g_i := \frac{\partial f}{\partial z_{i}} + \sum_{j = i + 1}^{n} g_{j} \left( \hat{\theta}_{j} \frac{\partial \eta_{j}}{\partial z_i} + ( 1 - \hat{\theta}_{j}) \frac{\partial \mu_{j}}{\partial z_i} \right),
$$
and $\hat{z} \gets \Forward{s, \hat{\theta}}$. Then there exists $Q \le 2 \beta B^{-2} (1+B)^{3n}$ s.t.
$$
\sum_{i=1}^{n} \abs{g_i - \frac{\partial \psi_i(s, \hat{z}_{1:i-1}, \hat{\theta}_{i+1:n})}{\partial \hat{z}_i}} \le Q \sum_{i=1}^{n} \abs{\theta_i - \hat{\theta}_i} \abs{\eta_i - \mu_i}.
$$
\end{lemma}

\begin{proof}
Denote $\eta_i(s, z_{1:i-1})$, $\mu_i(s, z_{1:i-1})$,  $\eta_i(s, \hat{z}_{1:i-1})$,  $\mu_i(s, \hat{z}_{1:i-1})$, $\psi_i(s, \hat{z}_{1:i-1}, \hat{\theta}_{i+1:n})$ by $\eta_i, \mu_i, \hat{\eta}_i$, $\hat{\mu}_i$, and $\hat{\psi}_i$ respectively. This proof is based on Fact~\ref{fact:a-minus-a-hat-bound}.
Define,
\begin{flalign*}
a_i &= \frac{\partial \psi_i}{\partial z_i},
\quad  
b_{i,j} = \hat{\theta}_{j} \frac{\partial \eta_{j}}{\partial z_i} + ( 1 - \hat{\theta}_{j}) \frac{\partial \mu_{j}}{\partial z_i},
\quad 
c_{i} = \frac{\partial f}{\partial z_j}
\\
\hat{a}_i &= \frac{\partial \hat{\psi}_i}{\partial \hat{z}_i}, \quad 
\hat{b}_{i,j} = \hat{\theta}_{j} \frac{\partial \hat{\eta}_{j}}{\partial z_i}  + ( 1 - \hat{\theta}_{j}) \frac{\partial \hat{\mu}_{j}}{\partial z_i}, \quad 
\hat{c}_{i} = \frac{\partial \hat{f}}{\partial z_j}.
\end{flalign*}
Moreover, observe that by our assumptions
\begin{subequations}\label{eq-group:bound-b-hat-c-hat}
\begin{flalign}
\abs{ b_{i,j} - \hat{b}_{i,j} } &\le
 \hat{\theta}_{j} \abs{ \frac{\partial \eta_{j}}{\partial z_i} - \frac{\partial \hat{\eta}_{j}}{\partial \hat{z}_i} } + ( 1 - \hat{\theta}_{j}) \abs{\frac{\partial \mu_{j}}{\partial z_i} - \frac{\partial \hat{\mu}_{j}}{\partial \hat{z}_i}}
\le \beta \| z - \hat{z} \|_1 \\
\| c - \hat{c} \|_{\infty} &\le \beta \| z - \hat{z} \|_1 \\
\max\{\| b \|_{\infty}, \| \hat{b} \|_{\infty} \} &\le B.
\end{flalign}
\end{subequations}
Therefore,
\begin{flalign*}
\| a - \hat{a} \|_1 &\le \left( \beta \| z - \hat{z} \|_1 + (1 + B)^{n} \beta \| z - \hat{z} \|_1 \right)
\frac{(1 + B)^{n}}{B} \\
&\le 2 \beta B^{-1} (1 + B)^{2n} \| z - \hat{z} \|_1 \\
&\le 2 \beta B^{-2} (1 + B)^{3n}  \sum_{i=1}^{n} \abs{\theta_i - \hat{\theta}_i} \abs{\eta_i - \mu_i} 
\end{flalign*}
where the first inequality uses \eqref{eq-group:bound-b-hat-c-hat} and Fact~\ref{fact:a-minus-a-hat-bound}, the second inequality uses $(1+B)^{n} \ge 1$, and the third uses Lemma~\ref{lem:bound-z-minus-z-hat}.
\end{proof}

Therefore, if Assumptions~\ref{assume:smooth} and  \ref{assume:bounded-set} hold then, by Fact~\ref{fact:smooth-implies-beta-smooth} and Lemma~\ref{lem:g-minus-psi-hat}, we conclude Assumption~\ref{assume:g-i-predict} holds.



\section{Implementation details}\label{app:implementation-details}

The fastest known way to derive bounds on the values taken by neural network is Interval Bound Propagation (IBP)~\citep{Gowal2019}.
It does not attempt to solve the optimization problem of the type of~\eqref{eq:convex-relaxation} which we define in Appendix~\ref{sec:certification-of-neural-networks} but simply performs interval analysis to derive loose bounds on the values taken by the neural network.
As a result, running IBP is a very fast procedure, whose cost is analogous to performing a forward pass through the network, but will give much looser bounds than solving the optimization problem.
We include it in Table~\ref{tab:benchmark} as a lower bound on the runtime that can be achieved by any of the method. 
We also use it to derive the intermediate bounds $l_i$ and $u_i$ required for the definition of the constraints~\eqref{eq:relu_cvx_hull} when formulating the optimization problems for the methods we are benchmarking.

For DeepVerify~\citep{Dvijotham2018}, we use the code released by the authors. The Lagrangian dual is optimized using a sub-gradient method.
We use the Adam optimizer~\citep{Kingma2015}, starting with an initial step size of $10^{-2}$ and decreasing it by a factor of 10 every 100 steps. We run the algorithm for a maximum of 500 steps, but allow early stopping if the dual objective value reaches the value returned by the dual objective of our NonConvex reformulation. In its original formulation, DeepVerify does not have access to a primal objective and can therefore not be provided with an appropriate stopping criterion. We provide it here with one in order to generate a more fair comparison,

Our NonConvex relaxation is solved following the principle of  \callSimplePsiMin{}. 
To optimize the objective function (line~\ref{algline:optstep} of Algorithm~\ref{alg:PGD}), we use the FISTA with backtracking linesearch~\citep{beck2009fast} algorithm.
We initialize our step size to a high value of 100, and at each optimization step, perform a line-search starting from the previous step-size, progressively shrinking it by a backtracking coefficient of 0.8 until sufficient progress can be guaranteed. 
In order to prevent the use of too small learning rates, we increase the step-size by a factor of 1.5 when no backtracking was necessary, and clip the step size to a minimum of $10^{-5}$.
We run a maximum of 50 steps of this algorithm, unless we first reach a situation where the relative gap between the primal and the dual objective becomes smaller than $10^{-2}$.

For the CVXPY~\citep{cvxpy} based solvers, we used the default settings.

\section{Benchmarking on Robustly trained Networks}\label{app:ibp_benchmarks}
In addition to the results reported in Section~\ref{sec:experiements}, we provide additional results for networks trained with Interval Bound Propagation (IBP)~\citep{Gowal2019}, to achieve robustness against $\mathcal{L}_\infty$-bounded perturbations of size $\epsilon=8.0/255$.
Because of the strong regularization effect of IBP, there is very little difference between the solution to the optimization problem described in Section~\ref{sec:certification-of-neural-networks} and the bounds produced by IBP, so better optimization algorithms do not prove as useful.

For these experiments, we didn't employ the relative duality gap criterion and instead optimized until the duality gap reached a value smaller than 1e-4.

Even on those networks for which the solution should be easier to obtain, we observe similar results.
Off-the-shelf solvers based on CVXPY do not scale beyond the Tiny network, and the runtime for DeepVerify is significantly larger than for our nonconvex solver, as shown in Table~\ref{tab:ibp_benchmark}.
The discrepancy in runtime comes from the fact that DeepVerify does not manage to converge to an accurate enough solution, as shown by the low percentage of early stopping in Table~\ref{tab:ibp_convergence_prop} and the observable plateaus in Figure~\ref{fig:ibp_opt_trace}.

We note that as opposed to networks trained with adversarial training, Figure~\ref{fig:ibp_distance_to_corner} reveals that when computing bounds on networks trained to achieve certifiable robustness, activations do get close to potentially degenerate points, which might be problematic for the optimization procedure and require us to employ Algorithm~\ref{alg:fix-deg} to find optimal solutions.


\begin{table}[h]
\floatbox[{\capbeside\thisfloatsetup{floatwidth=sidefil,capbesideposition={right,top},capbesidewidth=.32\linewidth}}]{table}
{\caption{\small Runtime of bound computation, in milliseconds. Results correspond to the computation of a lower bound on the gap between ground truth and all other classes for a network subject to an $L_\infty$ adversarial attack with budget $\frac{8.0}{255}$, averaged over the CIFAR-10 test set, on network trained with IBP.}\label{tab:ibp_benchmark}}%
{\small\addtolength{\tabcolsep}{-5pt}
\begin{tabular}{lcccccc}
        \toprule
        Methods & \multicolumn{3}{c}{ReLU activation} & \multicolumn{3}{c}{SoftPlus activation}\\
        & Tiny & Small & Medium & Tiny & Small & Medium\\
        \midrule
        IBP~\citep{Gowal2019} & 3.4 & 2.8 & 3.2 
                              & 3.9 & 2.3 & 2.8\\
        \midrule
        DeepVerify~\citep{Dvijotham2018} & 338 & 658 & 1171  
                   & 403 & 842 & 1620\\
        NonConvex (ours) & 93 & 176 & 297 
                         & 6.1 & 61 & 210\\
        \midrule
        Solver (SCS)~\citep{Odonoghue2016} & $78.4 \cdot 10^3$  & - & - 
                     & $32.1 \cdot 10^4$  & - & - \\
        Solver (ECOS)~\citep{domahidi2013ecos} & $41.2 \cdot 10^3$  & - & - 
                     & - & - & - \\
        \bottomrule
    \end{tabular}}
\end{table}

\begin{table}[h]
\floatbox[{\capbeside\thisfloatsetup{floatwidth=sidefil,capbesideposition={right,top},capbesidewidth=.32\linewidth}}]{table}
{\caption{\small Proportion of bound computations on CIFAR-10 where the algorithm converges within the iteration budget, and average number of iterations for each first order algorithm for an IBP trained network.}\label{tab:ibp_convergence_prop}}%
{\small\addtolength{\tabcolsep}{-5pt}
\begin{tabular}{l@{\hskip 0.05\linewidth}cc@{\hskip 0.05\linewidth}cc}
            \toprule
            & \multicolumn{2}{p{.25\linewidth}}{\centering Proportion of early stopping} & \multicolumn{2}{p{.25\linewidth}}{\centering Average number of iterations}\\
            & DeepVerify & NonConvex & DeepVerify & NonConvex \\ 
            \midrule
            Tiny ReLU & 92\% & 99\% & 242 & 6.5 \\
            Small ReLU & 25\% & 99\% & 479 & 5.5 \\
            Medium ReLU & 3\% & 100\%& 491 &  4.92 \\
            \midrule
            Tiny SoftPlus & 91 \% & 100\%& 238.33 & 7.4\\
            Small SoftPlus & 0\% & 99\%& 500 & 8.5\\
            Medium SoftPlus & 0\% & 98\%& 500 & 9.8\\
            \bottomrule
        \end{tabular}}
\end{table}

\begin{figure}[h]
    \begin{subfigure}{.62\textwidth}
        \includegraphics[width=\linewidth]{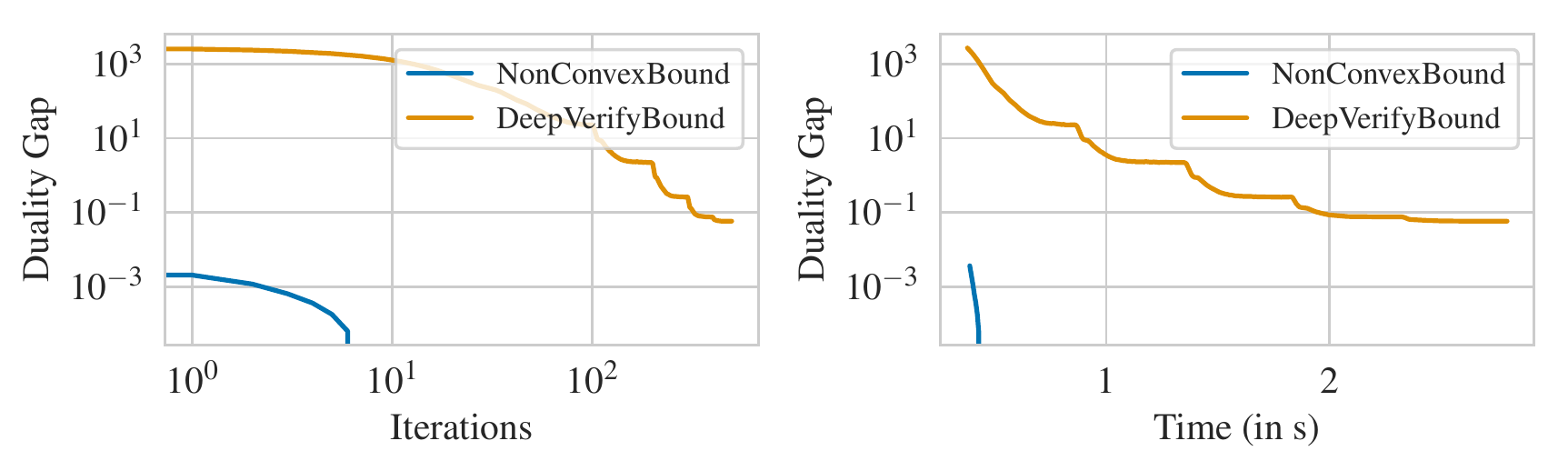}%
        \caption{Evolution of the duality gap as a function of time or number of iteration, for the NonConvex and DeepVerify Solver.}
        \label{fig:ibp_opt_trace}
    \end{subfigure}%
    \hfill%
    \begin{subfigure}{.31\textwidth}
        \centering
        \includegraphics[width=\linewidth]{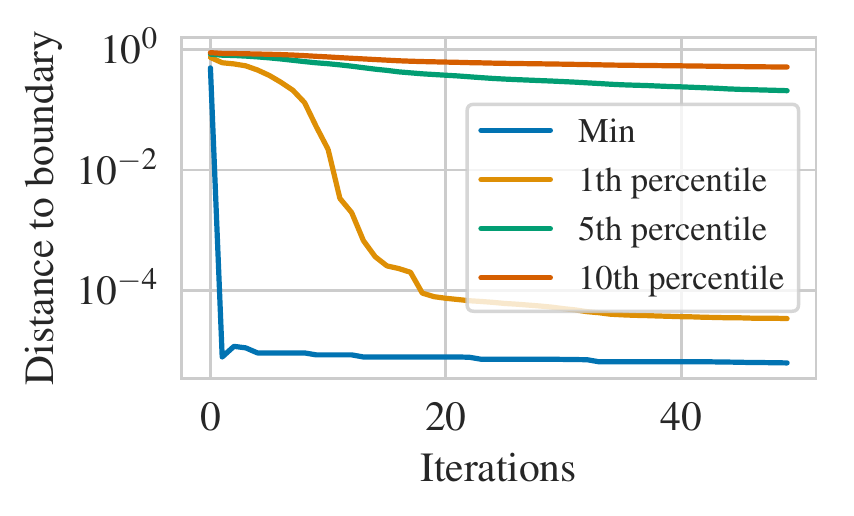}
        \caption{Distribution of distance to potentially degenerate point.}
        \label{fig:ibp_distance_to_corner}      
    \end{subfigure}%
    \hfill %
    \caption{Evaluation on the Medium-sized network with SoftPlus activation function trained with IBP.}
\end{figure}

\end{document}